\documentclass[10pt]{article}

\usepackage{amsmath,amsthm,amsfonts,amssymb,bm,bbm,enumerate, graphicx, mathtools,color}
\usepackage{tcolorbox}
\usepackage[english]{babel}
\usepackage[utf8]{inputenc}
\usepackage{fancyhdr}
\usepackage{caption}
\usepackage{subcaption}
\usepackage[toc,page]{appendix}
\usepackage{csquotes,multirow}
\usepackage{comment}

\definecolor{NBrown}{HTML}{66220C}
\definecolor{NAqua}{HTML}{00698C}
\definecolor{ForestGreen}{HTML}{228b22}

\usepackage{hyperref}
\usepackage[capitalise]{cleveref}

\parskip \medskipamount

\usepackage{mleftright}
\mleftright

\newtheorem{theorem}{Theorem}[section]
\newtheorem{lemma}[theorem]{Lemma}
\newtheorem{proposition}[theorem]{Proposition}
\newtheorem{corollary}[theorem]{Corollary}
\newtheorem{remark}[theorem]{Remark}
\newtheorem{claim}[theorem]{Claim}

\newtheorem{definition}[theorem]{Definition}

\newcommand{\pr}[1]{\mathbb{P}\!\left(#1\right)}
\newcommand{\E}[1]{\mathbb{E}\!\left[#1\right]}
\newcommand{\estart}[2]{\mathbb{E}_{#2}\!\left[#1\right]}
\newcommand{\prstart}[2]{\mathbb{P}_{#2}\!\left(#1\right)}

\newcommand{\prcond}[3]{\mathbb{P}_{#3}\!\left(#1\;\middle\vert\;#2\right)}

\newcommand{\tplus}{{\tmix^{+}}}
\newcommand{\Pb}[0]{\mathbb{P}}

\DeclareMathOperator{\LE}{LE}

\DeclareMathOperator{\LERW}{LERW}

\DeclareMathOperator{\UST}{UST}

\DeclareMathOperator{\diam}{Diam}
\DeclareMathOperator{\height}{Height}

\renewcommand{\epsilon}{\varepsilon}
\newcommand{\pand}[0]{\ \text{and} \ }

\DeclareMathOperator{\bin}{Bin}

\newcommand{\RR}{\mathbb{R}}

\newcommand{\XXbb}{\mathbb{X}}
\newcommand{\NN}{\mathbb{N}}

\newcommand{\T}{\mathcal{T}}

\newcommand{\Capp}{\mathrm{Cap}}
\newcommand{\Capk}{\mathrm{Cap}_k}

\newcommand{\tmix}{t_{\mathrm{mix}}}
\newcommand{\Xindj}{X^{\mathrm{ind}, j}}
\newcommand{\Xindone}{X^{\mathrm{ind}, 1}}
\newcommand{\Xindjj}{X^{\mathrm{ind}, J}}

\newcommand{\close}{\mathrm{Close}}

\newcommand{\eps}{\varepsilon}

\newcommand{\dGHP}{d_{\mathrm{GHP}}}
\newcommand{\dGP}{d_{\mathrm{GP}}}
\newcommand{\GHP}{\mathrm{GHP}}

\newcommand{\neweps}{\kappa}

\newcommand{\Tkn}{T^{(k)}_n}
\newcommand{\Tkkn}{T^{(k-1)}_n}
\newcommand{\Tkkkn}{T^{(k+1)}_n}
\newcommand{\Tk}{T^{(k)}}

\newcommand{\SBk}{\textsf{SB}^{(k)}}

\usepackage{tikz}

\usepackage{graphicx}
\usepackage{caption,subcaption}

\makeatletter
\renewcommand\@dotsep{10000}
\makeatother
\allowdisplaybreaks

\begin{document}

\title{The GHP scaling limit of uniform spanning trees of dense graphs}
\author{Eleanor Archer\thanks{\'Equipe Modal'X, Universit\'e Paris Nanterre, Batiment G, 200 Avenue de la R\'epublique, 92000 Nanterre, France. Email: \protect\url{eleanor.archer@parisnanterre.fr}} \hspace{1cm} Matan Shalev\thanks{School of Mathematical Sciences, Tel Aviv University, Tel Aviv, Israel. Email: \protect\url{matanshalev@mail.tau.ac.il}}}

	\maketitle
	\begin{abstract} 
	We consider dense graph sequences that converge to a connected graphon and prove that the GHP scaling limit of their uniform spanning trees is Aldous' Brownian CRT. Furthermore, we are able to extract the precise scaling constant from the limiting graphon. As an example, we can apply this to the scaling limit of the uniform spanning trees of the Erd\"os-R\'enyi sequence $(G(n,p))_{n \geq 1}$ for any fixed $p \in (0,1]$, and sequences of dense expanders. A consequence of GHP convergence is that several associated quantities of the spanning trees also converge, such as the height, diameter and law of a simple random walk.
	\end{abstract}
	\section{Introduction}

Uniform spanning trees (USTs) are fundamental objects in probability theory and computer science, with close connections to many other areas of mathematics including electrical network theory \cite{kirchhoff1847ueber}, loop erased random walks \cite{Wil96} and random interlacements \cite{hutchcroft2018interlacements}, to name but a few. 

It was recently shown in \cite{ANS2021ghp}, building on the work of \cite{PeresRevelleUSTCRT}, that the universal metric measure space scaling limit of USTs of a large class of graphs is Aldous' Brownian continuum random tree (CRT). The purpose of the present paper is to extend this result to sequences of dense graphs encoded by graphons. Due to a transitivity assumption in previous papers, these USTs are not covered by the results of \cite{PeresRevelleUSTCRT} and \cite{ANS2021ghp}, but here we establish that the CRT is nevertheless still the scaling limit. In addition we are able to express the precise scaling factor in terms of the encoding graphon, making the result more precise than that in \cite{ANS2021ghp} and demonstrating that the notion of graphon convergence is enough to fully determine the UST scaling limit.

The CRT, introduced by Aldous \cite{AldousCRTI, AldousCRTII, AldousCRTIII}, is a well-known object in probability theory, and is perhaps best-known as the scaling limit of critical finite variance Galton--Watson trees. We do not attempt to give a full introduction here; we will give a formal definition in Section \ref{sctn:stick breaking background} and we refer to the survey of Le Gall \cite{le2005random} for further background.

A \textbf{weighted graph} $(G, w)$ is a graph $G=(V,E)$ in which we assign to each edge $e\in E$ a non-negative weight $w_e$. In this paper, we will work with sequences of weighted graphs with no loops or multiple edges in which $w_e \in [0,1]$ for each $e\in E$. In the case where all edge-weights are equal to $1$, we say that the graph is \textbf{simple}. We extend the definition of vertex degree to weighted graphs by defining $\deg v$ to be the sum of the weights of the edges emanating from $v$.

The \textbf{uniform spanning tree} of a weighted graph $(G,w)$ is a random spanning tree chosen from the set of all spanning trees of $G$ where each spanning tree $t$ is chosen with probability proportional to $\prod_{e\in t} w_e$.

We will say that such a sequence $(G_n)_{n \geq 1}$ of weighted graphs is \textbf{dense} if there exists $\delta>0$ such that $\Delta_n := \min_{v \in G_n} \deg(v) \geq \delta \# V(G_n)$ for all $n$. The notion of convergence of dense graph sequences is naturally captured by objects known as \textit{graphons}, introduced by Lov{\'a}sz and Szegedy \cite{lovasz2006limits} and also Borgs, Chayes, Lov{\'a}sz, S{\'o}s and Vesztergombi \cite{borgs2008convergent} for this purpose. See also \cite{glasscock2015graphon} for a very quick introduction. A \textbf{graphon} $W$ is a symmetric measurable function from $[0,1]^2$ to $[0,1]$ and can be thought of as (roughly) the continuum analogue of an adjacency matrix. Using this viewpoint, there is a natural notion of distance between discrete graphs and graphons, known as the \textbf{cut-distance}, which we will define in \cref{sctn:graphon background}. This allows us to consider the notion of convergence to a given graphon $W$.

Graphons are commonly used in combinatorics and computer science to analyze large dense graphs. For example, they have been used in extremal graph theory \cite{cooper2018finitely}, mean-field games \cite{caines2019graphon}, analysis of large graphs \cite{krioukov2016clustering}, and to study the thermodynamic limit of statistical physics systems \cite{medvedev2014nonlinear, giardina2021approximating}, to give a very non-exhaustive list.

Given a graphon $W$, define a constant
\begin{equation}\label{eqn:alphaW}
\alpha_W = \frac{1}{\left(\int_{[0,1]^2}W(x,y)dxdy\right)^2} \cdot \int_{[0,1]} \left(\int_{[0,1]}W(x,y)dy\right)^2 dx.
\end{equation}

Note it follows immediately from Jensen's inequality that $\alpha_W \geq 1$, with equality if and only if $W$ is constant almost everywhere. We also say that a graphon $W$ is \textbf{connected} if for all $A \subset [0,1]$ of positive Lebesgue measure, it holds that
\begin{equation*}
    \int_{A}\int_{A^C}W(x,y)dxdy > 0.
\end{equation*}

The main result of the present paper is the following. Below, the GHP distance refers to the Gromov Hausdorff Prohorov distance between metric measure spaces; we define it in Section \ref{sctn:GHP def}.

\begin{theorem} \label{thm:main1}
Let $(G_n)_{n \geq 1}$ be a dense sequence of deterministic weighted graphs converging to a connected graphon $W$, where each $G_n$ has $n$ vertices. For each $n \geq 1$, let $\T_n$ be a uniform spanning tree of $G_n$. Denote by $d_{\T_n}$ the corresponding graph-distance on $\T_n$ and by $\mu_n$ the uniform probability measure on the vertices of $\T_n$. Then 
		\begin{equation*} \label{eq:mainthm1}
		\left(\T_n,\frac{\sqrt{\alpha_W}}{\sqrt{n}} d_{\T_n} ,\mu_n\right) \overset{(d)}{\longrightarrow} (\T,d_{\T},\mu)
		\end{equation*}
		where $\alpha_W$ is defined as in \eqref{eqn:alphaW}, $(\T,d_{\T},\mu)$ is the CRT equipped with its canonical mass measure $\mu$ and $\overset{(d)}{\longrightarrow}$ denotes convergence in distribution with respect to the GHP distance between metric measure spaces.
\end{theorem}

A single graphon can also encode sequences of \textit{random} graphs $G(k,W)_{k \geq 1}$ and $H(k,W)_{k \geq 1}$ with $k$ nodes, obtained by sampling $k$ uniform vertices $x_1, \dots, x_k$ in $[0,1]$, and either adding an edge of weight $1$ between nodes $i$ and $j$ with probability $W(x_i, x_j)$ (this is the sequence $G(k,W)_{k \geq 1}$), or instead adding an edge of weight $W(x_i,x_j)$ (this is the sequence $H(k,W)_{k \geq 1}$). We will deduce the following as a consequence of \cref{thm:main1}.

\begin{corollary}\label{cor:mainrandom}
Let $W$ be a connected graphon. Suppose that there exists $\delta>0$ such that the minimal degree of $G(n,W)$ is at least $\delta n$ with probability tending to $1$ as $n \to \infty$. For each $n \geq 1$, let $\T_n$ be a uniform spanning tree of $G(n,W)$. Denote by $d_{\T_n}$ the corresponding graph-distance on $\T_n$ and by $\mu_n$ the uniform probability measure on the vertices of $\T_n$. Then 
		\begin{equation*}
		\left(\T_n,\frac{\sqrt{\alpha_W}}{\sqrt{n}} d_{\T_n} ,\mu_n\right) \overset{(d)}{\longrightarrow} (\T,d_{\T},\mu)
		\end{equation*}
		where $(\T,d_{\T},\mu)$ is the CRT equipped with its canonical mass measure $\mu$ and $\overset{(d)}{\longrightarrow}$ denotes convergence in distribution with respect to the GHP distance between metric measure spaces.
		
		Moreover, the same statement holds for $H(n,W)$ in place of $G(n,W)$.
\end{corollary}

For example, this applies to the Erd\"os-R\'enyi sequence $(G(n,p))_{n \geq 1}$ for any fixed $p \in (0, 1]$, which is the sequence $(G(n,W))_{n \geq 1}$ when $W$ is the graphon that is $p$ (almost) everywhere, and in which case $\alpha_W=1$.

\cref{thm:main1} shows that graphons contain enough information to determine the scaling limit of USTs, or in other words that the GHP scaling limit is continuous with respect to the topology induced by the cut-distance. In \cite{hladky2018local}, the authors show an analogous result for the Benjamini-Schramm \textit{local limit} of the USTs appearing in \cref{thm:main1}, and show that the local limit can be characterized as a multi-type critical branching process conditioned to survive, where the offspring distributions are encoded by the limiting graphon. Additionally, the authors show that continuity also holds for the total \textit{number} of spanning trees of $G_n$, after being properly renormalized. However, they also give an example to show that this is no longer true under weaker assumptions.

Note that convergence of a graph sequence to a connected graphon does not automatically imply that the graph sequence must be dense, and in fact the local limit result for USTs of dense graphs obtained in \cite{hladky2018local} does not require this assumption. There, the authors assume only that the limiting graphon is \textbf{non-degenerate}, meaning that \begin{equation}\label{eqn: deg of graphon}
    \deg_W(x) := \int_{[0,1]} W(x,y) dy > 0 \ \ \forall x \in [0,1],
\end{equation}
and that the graph sequence is connected. In fact this implies that ``most" vertices have high degree; see \cite[Theorem 2.7 and Definition 2.6]{hladky2018local} for a precise statement. This is enough to prove a local limit statement since with high probability, the local limit will not see the exceptional vertices of low degree. On the other hand, the GHP scaling limit is a global statement and therefore we require more uniform control of the underlying graphs. One can easily see this through a simple counterexample: let $G_n$ denote the complete graph on $n-n^{2/3}$ vertices, and attach a stick of length $n^{2/3}$ to one vertex of the complete graph. The graphs still converge to the graphon that is $1$ everywhere, and the local limit of $\UST (G_n)$ is once again the \textsf{Poisson}($1$) Galton--Watson tree conditioned to survive. On the other hand, the only non-trivial compact scaling limit is a single stick, and not the CRT. One can also construct similar counterexamples with minimum degree at least $\frac{n}{\gamma_n}$ for any sequence $\gamma_n \to \infty$, meaning that the assumption of linear minimal degree is indeed necessary.

Since the local limit of the CRT is well-known \cite[Section 6]{AldousCRTI} to be Aldous' self-similar CRT (SSCRT), one can also ask whether the operations of taking scaling limits and local limits commute. In general, answering this question seems quite non-trivial, as the multitype branching process appearing as the local limit is very non-homogeneous and the offspring distributions of successive generations are not independent. However, a special case arises when the sequence $(G_n)_{n \geq 1}$ is regular. In this case the local limit is a \textsf{Poisson}($1$) Galton--Watson tree conditioned to survive, which is well-known to rescale to the SSCRT; moreover we will show in \cref{rmk:regular alpha=1} that the constant $\alpha_W$ must be equal to $1$, which entails that $\frac{1}{\alpha_W}$ is equal to the variance of the \textsf{Poisson}($1$) offspring distribution, and from which we can deduce that the operations do indeed commute in this case.

For non-regular graph sequences, the question seems a bit more subtle. While the expected number of non-backbone neighbours of the root vertex of the local limit is indeed $1$, the variance is not necessarily equal to $\frac{1}{\alpha_W}$. For example, for the complete bipartite graph $K_{\frac{2n}{3}, \frac{n}{3}}$, one can calculate using \cite[Definition 1.2]{hladky2018local} that the variance of the offspring number of the root vertex is equal to $\frac{3}{2}$, but $\frac{1}{\alpha_W}$ is equal to $\frac{8}{9}$. This does not preclude the possibility that the operations commute, since the variance in subsequent generations may converge to $\frac{1}{\alpha_W}$ in the appropriate sense. For $K_{\frac{2n}{3}, \frac{n}{3}}$ we can in fact apply results of Miermont \cite{miermontmulti} (the local limit in this case is in fact a Galton--Watson tree with two alternating types: \textsf{Poi}($2$) and \textsf{Poi}($\frac{1}{2}$)) to deduce that the operations do commute. However, in the general case the local limit is a Galton--Watson tree with uncountably many types, for which, to the best of our knowledge, scaling limits are not covered by the existing Galton--Watson tree literature.

Finally, we note that in \cite{alon2022diameter}, the authors consider similar dense graph sequences, but do not assume that the sequence converges to a graphon. Under this weaker assumption, they prove that the diameter of $\UST(G_n)$ is of order $\sqrt{n}$ with high probability. We cannot hope to prove a scaling limit result under the same hypotheses, since one can, for example, connect two copies of $K_{n/2}$ by a single edge, in which case the diameter is still of order $\sqrt{n}$ but the scaling limit is not the CRT. However, when the graphs are well-connected, we can obtain the scaling limit.

In this paper we in fact prove the following theorem. In what follows, for a given $\gamma>0$ we say that a graph $G$ is a $\gamma$-expander if for all $U \subset V(G)$, the number of edges between $U$ and $V(G) \setminus U$ is at least $\gamma|U|(|V(G)| - |U|)$.

\begin{theorem} \label{thm:main2}
Take $\gamma>0$ and $\delta>0$ and let $(G_n)_{n \geq 1}$ be a dense sequence of connected $\gamma$-expanders, where each $G_n$ has $n$ vertices and minimal degree at least $\delta n$. For each $n \geq 1$, let $\T_n$ be a uniform spanning tree of $G_n$. Denote by $d_{\T_n}$ the corresponding graph-distance on $\T_n$ and by $\mu_n$ the uniform probability measure on the vertices of $\T_n$. Then there exists a sequence $(\alpha_n)_{n \geq 1}$, satisfying $1 \leq \alpha_n \leq \delta^{-1}$ for all $n \geq 1$, such that
		\begin{equation*}
  \left(\T_n,\frac{\sqrt{\alpha_n}}{\sqrt{n}} d_{\T_n} ,\mu_n\right) \overset{(d)}{\longrightarrow} (\T,d_{\T},\mu)
		\end{equation*}
as $n \to \infty$ where $(\T,d_{\T},\mu)$ is the CRT equipped with its canonical mass measure $\mu$ and $\overset{(d)}{\longrightarrow}$ denotes convergence in distribution with respect to the GHP distance between metric measure spaces.
\end{theorem}

In fact the theorem holds slightly more generally, see \cref{rmk:generalisations}, but the above assumptions make the proof more straightforward. Clearly one cannot hope for convergence of the parameter $\alpha_n$ without making stronger assumptions, since one can alternate graphs from sequences with different limiting values of $\alpha_n$. For example, for the sequence of complete graphs $\alpha_n \to 1$, but if $G_n$ is instead the complete bipartite graph $K_{\frac{n}{3}, \frac{2n}{3}}$, then $\alpha_n \to \frac{9}{8}$.

As well as the convergence of the rescaled diameter, it follows directly from the GHP convergence of \cref{thm:main2} that we also have convergence of the rescaled height and rescaled simple random walk on $\UST (G_n)$. More formally, the following three convergences hold in distribution.

\begin{enumerate}
    \item $\frac{\sqrt{\alpha_n}\diam (\T_n)}{\sqrt{n}} \overset{(d)}{\to} \diam (\T)$.
    \item $\frac{\sqrt{\alpha_n}\height (\T_n)}{\sqrt{n}} \overset{(d)}{\to} \height (\T)$.
    \item If $X_n$ is a simple random walk on $\T_n$, then the quenched law of $\left(\frac{\sqrt{\alpha_n}}{\sqrt{n}}X_n({2n^{3/2}\alpha_n^{-1/2}}t)\right)_{t \geq 0}$ converges in distribution to the quenched law of Brownian motion on the CRT. It also follows that the associated mixing times converge on the same time scale.
\end{enumerate}

See \cite[Section 1.3]{ANS2021ghp} for further details of why these three properties follow from GHP convergence. In the settings of \cref{thm:main1} and \cref{cor:mainrandom}, we can replace $\alpha_n$ with $\alpha_W$ in the above three statements.

\subsection{Proof strategy}

Clearly, in order to prove the main theorems, it suffices to first prove \cref{thm:main2} and then show that the graph sequence is an expander sequence and that $\alpha_n \to \alpha_W$ under the additional assumption of \cref{thm:main1}.

We will prove \cref{thm:main2} in two steps using the lower mass bound criterion of \cite{AthreyaLohrWinterGap}. In particular, by \cite[Theorem 6.5]{ANS2021ghp}, in order to prove the GHP convergence of \cref{thm:main2} it is enough to prove the following two statements.
\begin{enumerate}[(A)]
    \item The convergence holds in a finite-dimensional sense (this will be formally stated in \cref{thm:pr04}).
    \item The lower mass bound condition holds; that is, if $m_n(\eta) = \inf_{x \in \UST (G_n)} \left\{\frac{|B(x,\eta \sqrt{n})|}{n}\right\}$, then for every $\eta>0$ the sequence $m_n(\eta)^{-1}$ is tight (this will be formally stated in \cref{clm:lmb}).
\end{enumerate}

The second condition will follow quite straightforwardly from minor adaptations of the arguments in \cite{ANS2021ghp}. The bulk of this paper is devoted to proving the first condition. In fact, this condition is equivalent to the joint convergence, for all $k \geq 1$, of the set of $k \choose 2$ distances between $k$ points chosen uniformly at random in $\UST (G_n)$. 

This type of convergence was previously proved for USTs of sequences of high-dimensional graphs in \cite{PeresRevelleUSTCRT}. This is a different class of graphs and includes the assumption of transitivity. Their proof uses Wilson's algorithm, which is a method for sampling USTs one branch at a time by running loop erased random walks (LERWs). In their proof, they couple Wilson's algorithm on $G_n$ with Wilson's algorithm on the complete graph and prove that the set of $k \choose 2$ distances on the two graphs must have the same scaling limit.

Our proof, by contrast, is more direct. We also use Wilson's algorithm, but we work directly with $\UST (G_n)$ and use the Laplacian random walk representation of LERWs to sample each branch. By tightly controlling the capacity of loop-erased random walks, we are able to directly compute the probability that a given branch exceeds a given length, and show that this converges to the analogous quantity for the CRT using Aldous' stick-breaking construction.

\begin{remark}\label{rmk:generalisations}
As demonstrated by the examples and discussion above \cref{thm:main2}, the assumption of linear minimal degree is necessary in order to obtain convergence in the GHP topology. In order to keep the exposition clean, we prove both conditions (A) and (B) above under these assumptions. However, the assumption is not really necessary for condition (A). The proof would work unchanged if we allow $o(n)$ vertices to have degrees less than $\sqrt{n}$, for example (since the loop-erased random walk that we analyze in \cref{sctn:Wilson analysis} will never hit this set, whp). In fact, we believe that it may be possible to adapt our proof of condition (A) (\cref{thm:pr04}) to work under the original assumptions of \cite{PeresRevelleUSTCRT}, but this would require one to keep track of several additional messy details, and would not add further insight.
\end{remark}

\subsection{Organization of the paper}
This paper is organized as follows. In \cref{sctn:background} we give the necessary background, including an introduction to graphons, USTs and the topologies of interest. In \cref{sctn:stick breaking background} we introduce a general framework for stick-breaking constructions of trees, and state Aldous' stick-breaking construction of the CRT. In \cref{sctn:RW estimates} we give some precise random walk estimates and we apply these with the Laplacian random walk representation in \cref{sctn:Wilson analysis} to obtain estimates for the first steps of Wilson's algorithm. In \cref{sctn:Wilson is stickbreaking} we use these estimates to couple stick-breaking on the CRT with Wilson's algorithm and prove that the two processes are very similar when $n$ is large enough. This proves condition (A) above. We also explain how (B) can be deduced from the results of \cite{ANS2021ghp} which in fact establishes \cref{thm:main2}. Finally, in \cref{sctn:graphon conv} we prove \cref{thm:main1} and \cref{cor:mainrandom}.

\subsection{Acknowledgments}
We would like to thank Asaf Nachmias and Jan Hladky for suggesting to look at graphons and for many helpful comments. This research is supported by ERC consolidator grant 101001124 (UniversalMap), and by ISF grant 1294/19. EA was partially supported by the ANR ProGraM grant.

\section{Background}\label{sctn:background}

\subsection{Graphons}\label{sctn:graphon background}
A \textbf{graphon} is a symmetric measurable function $[0,1]^2 \to [0,1]$. As mentioned in the introduction, graphons were introduced by Borgs, Chayes, Lov{\'a}sz, S{\'o}s, Szegedy and Vesztergombi \cite{lovasz2006limits, borgs2008convergent} in order to characterize dense graph limits. To understand why this definition is natural, we define the \textbf{graphon representation} of a discrete graph $G$ as follows. Suppose that $G$ is a simple graph with $n$ vertices. Number the vertices from $v_1$ to $v_n$ and partition the interval $[0,1]$ into a sequence of intervals $(I_i)_{i=1}^n$, where $I_i = \left[\frac{i-1}{n}, \frac{i}{n}\right]$ for each $1 \leq i \leq n$. We define the graphon $W_G: [0,1]^2 \to [0,1]$ by (e.g. see \cite[Section 7.1]{lovasz2006limits})
\[
W_G((x,y)) = \mathbbm{1}\{v_{\lceil nx \rceil \vee 1} \sim v_{\lceil ny \rceil \vee 1}\} \hspace{1cm} \forall \ (x, y) \in [0,1]^2.
\]

If $G$ is a \textit{weighted} graph, we instead define
\[
W_G((x,y)) = w(v_{\lceil nx \rceil \vee 1}, v_{\lceil ny \rceil \vee 1}) \hspace{1cm} \forall \ (x, y) \in [0,1]^2,
\]
where $w(v_i, v_j)$ represents the weight of the edge joining $v_i$ and $v_j$ (and is zero if there is no such edge).

Note that, given only $G$, this definition of $W_G$ is not unique, since it depends on the ordering of the vertices. Therefore, in order to define a metric on the space of graphons, we will instead consider equivalence classes of graphons. In particular, given two graphons $W_1$ and $W_2$ the \textbf{cut-distance} between them is defined as (e.g. see \cite[Equation (8.16)]{lovasz2006limits})
\begin{equation*}
    \delta_{\square} (W_1, W_2) = \inf_{\varphi} ||W_1^{\varphi} - W_2||_{\square},
\end{equation*}
where the infimum is taken over all measure-preserving automorphisms of $[0,1]$, where $W^{\varphi}$ is defined by $W^{\varphi}(x,y)=W(\varphi(x), \varphi(y))$, and where the \textbf{cut-norm} of a measurable function $U:[0,1]^2 \to [-1,1]$ is given by 
\begin{equation*}
    ||U||_{\square} = \sup_{S,T \in \mathcal{B}([0,1])} \left| \int_{x \in S} \int_{y \in T} U(x,y) dx dy \right|.
\end{equation*}

We therefore say that a sequence of deterministic graphs $(G_n)_{n \geq 1}$ \textbf{converges to a graphon} $W$ if $\delta_{\square}(W_{G_n}, W) \to 0$ as $n \to \infty$.

\begin{remark}
Graphons can in fact be defined as functions from $\Omega^2 \to [0,1]$, where $\Omega$ is any probability space, see \cite[Chapter 13]{lovasz2006limits}, but since all probability spaces are isomorphic, this does not provide much greater generality.
\end{remark}

We will make use of the following lemma.

\begin{lemma}\label{lem:connectivity}\cite[Lemma 7]{bollobas2010percolation}.
Let $W$ be a connected graphon. Then, for every $\alpha \leq 1/2$ there exists some constant $\beta = \beta(W,\alpha)$ such that for every set $A$ with $\alpha \leq \mu(A) \leq 1/2$ we have
\begin{equation*}
    \int_{A}\int_{A^C} W(x,y) dx dy > \beta.
\end{equation*}
\end{lemma}

\subsubsection{Random graphs and graphons}
A graphon $W$ can be used to define a random graph with $n$ vertices in two ways.
\begin{enumerate}
    \item Sample $x_1, \ldots, x_n$ i.i.d. uniformly on $[0,1]$. We define a \textbf{random simple graph} on $\{1, \ldots, n\}$ by joining nodes $i$ and $j$ with probability $W(x_i, x_j)$, independently for each (unordered) pair $(i,j)$. We denote the resulting random graph $G(n, W)$.
    \item Sample $x_1, \ldots, x_n$ i.i.d. uniformly on $[0,1]$. We define a \textbf{random weighted graph} on $\{1, \ldots, n\}$ by adding an edge between $i$ and $j$ of weight $W(x_i, x_j)$ for each (unordered) pair $(i,j)$. We denote the resulting random graph $H(n, W)$.
\end{enumerate}
In both constructions, note that we can use a \textit{single} graphon to define a whole sequence of random graphs. The following lemma tells that in either case, the cut-distance between a random sample of $G(k,W)$ or $H(k,W)$ and $W$ goes to zero w.h.p. as $k \to \infty$.

\begin{lemma}\cite[Lemma 10.16]{lovasz2006limits}.\label{lem:random graphons are close}
Fix a graphon $W$ and for $k \geq 1$, let $G(k,W)$ and $H(k,W)$ be defined as above. Then, $\delta_{\square}(W_{G(k,W)}, W)$ and $\delta_{\square}(W_{H(k,W)}, W)$ both tend to $0$ in probability as $k \to \infty$.
\end{lemma}

In particular this means that results we prove for USTs of deterministic sequences of graphs extend automatically to sequences of the form $G(k,W)_{k \geq 1}$ or $H(k,W)_{k \geq 1}$ under the assumptions of \cref{cor:mainrandom}.

For example, the classical Erd\"os-R\'enyi graphs $\mathcal{G}(n,p)$ for $n \geq 1$, $p \in [0,1]$ correspond to the graphs $G(n, W_p)$ where $W_p$ is the graphon that is equal to $p$ everywhere.

For further background and applications of graphons, we refer to \cite[Part 3]{lovasz2006limits}.

\subsection{Mixing times} 
Let $G$ be a connected weighted graph with $n$ vertices, with weights $(w(x,y))_{x, y \in V(G)}$, and with no loops or multiple edges. A \textbf{random walk} on $G$ is the Markov Chain $(X_m)_{m \geq 0}$ such that, for all vertices $x, y \in V(G)$, and all $m \geq 1$,
\[
\prcond{X_m = y}{X_{m-1}=x}{} = \frac{w(x,y)}{\sum_{z \sim x} w(x,z)},
\]
where $z \sim x$ means that $z$ is a neighbour of $x$. Due to periodicity considerations, it is sometimes more convenient to instead use the notion of a \textbf{lazy random walk}. This is defined by
\[
\prcond{X_m = y}{X_{m-1}=x}{} = \frac{w(x,y)}{2\sum_{z \sim x} w(x,z)} \ \forall y \sim x \ \ \text{and} \ \ \prcond{X_m = x}{X_{m-1}=x}{} = \frac{1}{2}
\]
for all $m \geq 1$.

For each $t \geq 0$ let $p_t$ denote the $t$-step transition density of a lazy random walk, i.e. $p_t (x,y) = \prcond{X_t = y}{X_{0}=x}{}$ for all $x, y \in V(G)$. We define the {\bf mixing time} of $G$ as
	\begin{equation} \label{def:tmix}
	\tmix(G) = \min\left\{t\geq 0: \max_{x,y\in G} \left|  p_t(x,y)  - \pi(x)\right| \leq \frac{1}{4}\right\}, 
	\end{equation}
 (see \cite[Equation (4.31)]{levin2017markovmixing}),
where $\pi$ denotes the stationary measure on $G$. 

We will also need the notion of {\bf total variation distance} between two probability measures on $\mu$ and $\nu$ on a finite subset $X \subset V(G)$. This is defined by
	$$ d_{\textrm{TV}}(\mu,\nu) = \max_{A \subset X} |\mu(A)-\nu(A)| \, .$$
Furthermore, by \cite[Section 4.5]{levin2017markovmixing}, we have for any $k \geq 1$, any $t \geq k\tmix$ and any vertex $x$ that
	\begin{equation}\label{eq:tvDecay}
	d_{\textrm{TV}}(p_t(x,\cdot),\pi(\cdot)) \leq 2^{-k}.
	\end{equation}

\subsection{Expanders}

We will use the following definition of an expander graph.

\begin{definition}(\cite[Definition 2.1]{hladky2018local}).\label{def:expander}
For any $\gamma > 0$, a loopless weighted graph $G$ is a $\gamma$-expander if for all $U \subset V(G)$, we have that $w(U, V(G) \setminus U) \geq \gamma|U|(V(G) - |U|)$ where $w(A,B) = \sum_{v\in A, u\in B} w(v,u)$.
\end{definition}

Although we give the definition for loopless graphs, note that adding loops to a graph does not change the law of its UST, since loops can never appear in a UST. Note that often in the literature a slightly different definition of expander is used, involving the Cheeger constant. We are using the definition above as it fits more naturally into the framework of dense graphs (as we will later show in \cref{claim:graphon is an expander}) and is the same definition used to consider the local limit in \cite{hladky2018local}.

The main property of expanders that we will use is as follows.

\begin{claim}\label{claim:mixing time expander}
Let $\gamma>0$ and let $G$ be a $\gamma$-expander with $n\geq 2$ vertices. Then, provided that $n$ is large enough (depending on only $\gamma$), we have that
\[
\tmix (G) \leq \frac{64}{\gamma^4}\log n.
\]
\end{claim}
\begin{proof}
Note that it follows from \cref{def:expander} that $G$ has minimal degree at least $\frac{\gamma}{2}n$. First note that by \cite[Theorem 12.4]{levin2017markovmixing} that
\[
\tmix (G) \leq t_{\text{rel}} \log \left( \frac{8n}{\gamma}\right),
\]
where $t_{\text{rel}}$ is the relaxation time of $G$. By the Cheeger inequality (see \cite{alon1986eigenvalues, alon1985lambda1, jerrum2004elementary, lawler1988bounds} for various proofs), $\frac{1}{t_{\text{rel}}}$ is lower bounded by $\Phi(G)^2/2$, where
\[
\Phi(G) =  \min_{S \subset V(G), \pi(S) \leq 1/2} \frac{w(S, V(G)\setminus S)}{\sum_{v \in S} \deg v}.
\]
Note that $\pi(S) \leq 1/2$ implies that $(|V(G)|- |S|) \geq \frac{n\gamma}{4}$. Since $G$ is a $\gamma$-expander, it follows that 
\[
\Phi(G) \geq \frac{\gamma |S|(|V(G)|- |S|)}{\sum_{v \in S} \deg v} \geq \frac{\gamma |S|(|V(G)|- |S|)}{|S|n} \geq \frac{\gamma^2}{4}.
\]
Combining all the inequalities gives the result.
\end{proof}

\subsection{Loop-erased random walk and Wilson's algorithm}\label{sctn:Wilson}
We now describe Wilson's algorithm \cite{Wil96} which is a widely-used algorithm for sampling $\UST$s.
A {\bf walk} $X=(X_0, \ldots X_L)$ of length $L\in\mathbb{N}$ is a sequence of vertices where $(X_i, X_{i+1})\in E(G)$ for every $0 \leq i \leq L-1$. For an interval $J=[a,b]\subset[0,L]$ where $a,b$ are integers, we write $X[J]$ for $\{X_i\}_{i=a}^{b}$. Given a walk $X$, we define its {\bf loop erasure} $Y = \LE(X) = \LE(X[0,L])$ inductively as follows. We set $Y_0 = X_0$ and let $\lambda_0 = 0$. Then, for every $i\geq 1$, we set $\lambda_i = 1+\max\{t \mid X_t = Y_{\lambda_{i-1}}\}$ and if $\lambda_i \leq L$ we set $Y_i = X_{\lambda_i}$. We halt this process once we have $\lambda_i > L$. When $X$ is a random walk on the weighted graph $G$ starting at some vertex $v$ and terminated when hitting another vertex $u$ ($L$ is now random), we say that $\LE(X)$ is a {\bf loop erased random walk} ($\LERW$) from $v$ to $u$.  
	
To sample a $\UST$ of a finite connected weighted graph $G$ we begin by fixing an ordering of the vertices of $V=(v_1,\ldots, v_n)$. First let $T_1$ be the tree containing $v_1$ and no edges. Then, for each $i>1$, sample a $\LERW$ from $v_i$ to $T_{i-1}$ and set $T_i$ to be the union of $T_{i-1}$ and the $\LERW$ that has just been sampled. We terminate this algorithm with $T_n$. Wilson \cite{Wil96} proved that $T_n$ is distributed as $\UST(G)$. An immediate consequence is that the path between any two vertices in $\UST (G)$ is distributed as a $\LERW$ between those two vertices. This was first shown by Pemantle \cite{Pem91}.

\subsection{Laplacian random walk}\label{sctn:Laplacian RW}
Here we outline the Laplacian random walk representation of the LERW (see \cite[Section 4.1]{LyonsPeres} for full details) and its application to Wilson's algorithm. Take a finite, weighted, connected graph $G$ and suppose we have sampled $T_j$ for some $j \geq 1$ using Wilson's algorithm as described above. We now sample a LERW from $v_{j+1}$ to $T_j$. Denote this LERW by $(Y_m)_{m \geq 0}$. Also let $X$ denote a random walk on $G$. For a set $A \subset G$, let $\tau_A$ denote the hitting time of $A$ by $X$, and $\tau_A^+$ denote the first return time to $A$ by $X$. The Laplacian random walk representation of $Y$ says that, conditionally on $T_j$ and on the event $\{(Y_m)_{m=0}^{i} \cap T_j = \emptyset\}$, we have for any $i \geq 0$ that 
\begin{align*}
\prcond{ Y_{i+1} = v}{(Y_m)_{m=0}^{i}}{} = \prcond{ X_{1} = v}{\tau_{T_j} < \tau^+_{\cup_{m=0}^{i} \{Y_m\}}}{Y_i} = \frac{\prstart{ X_{1} = v}{Y_i} \prstart{\tau_{T_j} < \tau_{\cup_{m=0}^{i} \{Y_m\}}}{v}}{\prstart{\tau_{T_j} < \tau^+_{\cup_{m=0}^{i} \{Y_m\}}}{Y_i}}. 
\end{align*}
Clearly this is only non-zero when $v \notin \bigcup_{m=0}^{i} \{Y_m\}$. We can now extrapolate this to ask about the law of $(Y_m)_{m=i+1}^{i+H}$ for some $H \geq 1$, given $(Y_m)_{m=0}^{i}$. In particular, if $u_0, u_1, \ldots, u_H$ is a simple path in $G_n$, where $\{u_1, \ldots, u_{H-1}\}$ is disjoint from $\bigcup_{m=0}^{i} \{Y_m\} \cup T_j$ and $u_0 = Y_i$, then
\begin{align*}\label{eqn:LERW path factorisation intro}
\begin{split}
\prcond{ (Y_m)_{m=i+1}^{i+H} = (u_m)_{m=1}^{H}}{(Y_m)_{m=0}^{i}}{}
&= \prstart{ (X_m)_{m=1}^{H} = (u_m)_{m=1}^{H}}{u_0} C((Y_m)_{m=0}^{i}, T_j, (u_m)_{m=1}^{H})),
\end{split}
\end{align*}
where 
\begin{equation*}\label{eqn:C constant expression intro}
C((Y_m)_{m=0}^{i}, T_j, (u_m)_{m=1}^{H}) = \prod_{h=1}^{H} \frac{\prstart{\tau_{T_j} < \tau_{\cup_{m=0}^{i} \{Y_m\} \cup \ \cup_{m=1}^{h-1} \{u_m\}}}{u_h}}{\prstart{\tau_{T_j} < \tau^+_{\cup_{m=0}^{i} \{Y_m\} \cup \ \cup_{m=1}^{h-1} \{u_m\}}}{u_{h-1}}}.
\end{equation*}

\subsection{Capacity and closeness}\label{sctn:capacity defs}
	
Recall that $G$ is a connected weighted graph with $n$ vertices with minimal degree at least $\delta n$. The capacity of a set of vertices of $G$ quantifies how difficult it is for a random walk to hit the set. Let $(X_i)_{i\geq 0}$ be a random walk on $G$ and for $U \subset V(G)$, let $\tau_U = \inf \{i \geq 0: X_i \in U\}$. Given $k \geq 0$ we define the $\mathbf{k}$\textbf{-capacity} of $U$ by $\Capk(U) = \prstart{\tau_U \leq k}{\pi}$.

Here we collect some useful facts about the capacity.

\begin{lemma}\label{lem:cap union bound}
	Let $A \subset V(G)$ and $k \geq 1$. Then
	\begin{equation}\label{eqn:cap UB}
	\Capk (A) \leq k\pi(A) \leq \frac{k|A|}{\delta n}.
	\end{equation}
Moreover, if $k|A| \leq \frac{\delta^3 n}{2}$, then
\begin{equation}
    \Capk (A) \geq \frac{k\pi(A)}{2} \geq \frac{\delta k|A|}{2n} .
\end{equation}
	\end{lemma}
	\begin{proof}
	The upper bound follows from a union bound. The lower bound follows from the Bonferroni inequalities and the lower bound on the degree, which imply that
	\begin{equation*}
    \Capk (A) \geq k\pi(A) - \left(\frac{k|A|}{\delta n}\right)^2 \geq \frac{\delta k|A|}{2n}. \qedhere
\end{equation*}
	\end{proof}
	
	We will also use the following claim.	
	\begin{claim}\label{claim:capacity different start point}
	Let $\tmix = \tmix(G)$. Let $A \subset V(G)$, let $M \geq (\log n)^2 \tmix$ and suppose that $(\log n)^2 \cdot \tmix|A| \leq n$. Then, provided $n$ is large enough,
	\[
	\sup_{u \in V(G) \setminus A} \left|\prstart{\tau_A \leq M}{u} - \Capp_M(A) \right| \leq \frac{3\log n \cdot \tmix |A|}{\delta n}.
	\]
	\end{claim}
	\begin{proof}
	Let $X$ be a random walk started at $u \in G$. Clearly, for any $t \geq 0$, the first $t$ steps of $X$ can be coupled with the first $t$ non-repeat steps of a lazy random walk $\Tilde{X}$. Therefore, first run a lazy random walk started from $u$ until time $T=2\log n \cdot \tmix$. Let $N$ denote the total number of non-repeat jumps of this lazy random walk. The distribution of $\tilde{X}_t$ is almost stationary by \eqref{eq:tvDecay}. Moreover, we have that $0 \leq N \leq T$ deterministically. To sample $(X_t)_{t = 0}^M$, we first couple it with the first $N$ steps of $(\tilde{X}_t)_{t = 0}^T$ as explained above, and then run $X$ for a further $M-N$ steps. Under this coupling, we therefore have from a union bound that
	\[
	\prstart{\tau_A \leq M}{u} \leq \frac{2\log n \cdot \tmix |A|}{\delta n} + \prstart{\tau_A \leq M}{\pi} + 2^{-2\log n} \leq \Capp_M(A) + \frac{3\log n \cdot \tmix |A|}{\delta n}.
	\]
	Similarly,
	\[
	\prstart{\tau_A \leq M}{u} \geq \prstart{\tau_A \leq M-T}{\pi} - 2^{-2\log n} \geq \Capp_M(A) - \frac{3\log n \cdot \tmix |A|}{\delta n}.
	\]

	\end{proof}

In order to obtain lower bounds on capacity, we define the $\mathbf{k}$\textbf{-closeness} of two sets $U$ and $W$ by 
	\begin{equation}\label{def:close}
	\close_k (U,W) = \prstart{\tau_{U} < k, \tau_{W} < k}{\pi}.
	\end{equation}
	
\begin{corollary}\label{cor:closeness UB deter}
    For any disjoint sets $U, W \subset G$, we have that
	\begin{equation*}
	\sup_{v \in G \setminus (U \cup W)} \prstart{\tau_{U} < k, \tau_{W} < k}{v} \leq \frac{2k^2|U||W|}{\delta^2 n^2}.
	\end{equation*}
In particular, $\close_k (U, W) \leq \frac{2k^2|U||W|}{\delta^2 n^2}$.
\end{corollary}	
\begin{proof}
Note that
\begin{align*}
\sup_{v \in G \setminus (U \cup W)} \prstart{\tau_{U} < k, \tau_{W} < k}{v} &\leq \sup_{v \in G \setminus (U \cup W)} \left\{ \prstart{\tau_U < \tau_W <k}{v} + \prstart{\tau_W < \tau_U <k}{v}\right\} \\
 &\leq \sup_{v \in G \setminus (U \cup W), u \in U, w \in W} \left\{ \prstart{\tau_U <k}{v}\prstart{\tau_W <k}{u} + \prstart{\tau_W <k}{v} \prstart{\tau_U <k}{w} \right\} \\&\leq \frac{2k^2|U||W|}{(\delta n)^2}.\qedhere
\end{align*}
\end{proof}

\subsection{Random variables}
Here we present two elementary results that will be useful in \cref{sctn:Wilson is stickbreaking}.

\begin{claim}\label{cl: two uniforms are close}
Let $\eps>0$ and let $0<a<b$ with $b-a \leq \eps$. Let $X_a \sim U\left(\left[0,a\right]\right)$ and $X_b \sim U\left(\left[0,b\right]\right)$. Then, we can couple $X_a$ and $X_b$ such that $\pr{|X_a - X_b| > \eps} < \eps$. 
\end{claim}
\begin{proof}
We take $X_b = \frac{b}{a}X_a$. Then, $|X_b - X_a| = |\frac{b-a}{a}\cdot X_a| \leq |b-a| \leq \eps$.	
\end{proof}

	\begin{lemma}\label{lem: proh close rv}
	For any $L>0$, let $X_L$ be the random variable on $(0, \infty)$ satisfying 
	\[
	\pr{X_L > x} = \exp\left\{-\frac{(x+L)^2-L^2}{2}\right\}.
	\]
	Then for any $\delta > 0$, there exists $\eta = \eta (\delta, L)>0$ such that the following holds. Let $Y$ be another random variable on $(0, \infty)$, and suppose that for all $x>0$,
	\begin{equation}\label{eqn:variables close}
	|\pr{X_L > x} - \pr{Y > x}|<\eta.
	\end{equation}
	Then this implies that we can couple $X_L$ and $Y$ so that $\pr{|X_L - Y|>\delta} < \delta$.
	
	Furthermore, for any $\delta, L_1$ and $L_2$ with $L_1< L_2$, there exists $\eta = \eta(\delta, L_1, L_2)$ such that we can couple $X_L$ and $Y$ as described above for every $L\in [L_1,L_2]$.
	\end{lemma}
	\begin{proof}
	Note that we can couple $X_L$ and $Y$ by first sampling $U \sim \textsf{Uniform}([0,1])$ and setting 
	\begin{align*}
	    X_L (\omega) = \sup_{x \geq 0} \{ \pr{X_L \geq x} \geq U(\omega)\}, \hspace{1cm} Y (\omega) = \sup_{x \geq 0} \{ \pr{Y \geq x} \geq U(\omega)\}.
	\end{align*}
	Now choose $K_{\delta, L}<\infty$ so that $\pr{X_L \geq K_{\delta,L}}<\delta$. Wlog assume that $\delta<1$ and $K_{\delta, L}>1$, otherwise decrease or increase them if necessary. Note that, for all $0 \leq x< K_{\delta, L}$, we have that
	\begin{align*}
	   \exp\left\{-\frac{(x+L)^2-L^2}{2}\right\} - \exp\left\{-\frac{(x+\delta+L)^2-L^2}{2}\right\} \geq \delta (x+L)\exp\left\{-\frac{(x+\delta+L)^2-L^2}{2}\right\} \geq  M_{\delta, L},
	\end{align*}
	where $M_{\delta, L} = \delta L\exp\left\{-\frac{(2K_{\delta, L}+L)^2-L^2}{2}\right\} > 0$.
	
	Now suppose that \eqref{eqn:variables close} holds and $\eta < M_{\delta, L}$. Then, for any $0 \leq x< K_{\delta, L}$ we have that
	\begin{align*}
	    \pr{Y \geq x+\delta} \leq \pr{X_L \geq x+\delta} + \eta \leq \pr{X_L \geq x} - M_{\delta, L} + \eta \leq \pr{X_L \geq x}.
	\end{align*}
	Therefore, under the coupling, we have for any $x< K_{\delta, L}$ that
	\[
	\{X_L \leq x\} \Leftrightarrow \{ \pr{X_L \geq x} \leq U \} \Rightarrow \{ \pr{Y \geq x+\delta} \leq U \} \Leftrightarrow \{Y \leq x+\delta\}.
	\]
	Similarly, $\{X_L \geq x\} \Rightarrow \{Y \geq x-\delta\}$. Therefore, under this coupling we have that 
	\begin{align*}
	    \pr{|X_L - Y|>\delta} \leq \pr{X_L \geq K_{\delta,L}} < \delta,
	\end{align*}
	as required.
 
 For the second claim, note that for every $L' > L$ we also have that $\pr{X_{L'} \geq K_{\delta,L}}<\delta$. Therefore for the interval $[L_1,L_2]$ we can simply use $K_{\delta, L_1}$ and $M_{\delta, L_1}$ on the whole interval.
	\end{proof}

\subsection{GHP topology}\label{sctn:GHP def}

Here we define the GHP topology. We use the framework of \cite[Sections 1.3 and 6]{MiermontTessellations} and work in the space $\XXbb_c$ of equivalence classes of metric measure spaces (mm-spaces) $(X,d,\mu)$ such that $(X,d)$ is a compact metric space and $\mu$ is a Borel probability measure on it, and we say that $(X,d,\mu)$ and $(X',d',\mu')$ are equivalent if there exists a bijective isometry $\phi: X \to X'$ such that $\phi_* \mu = \mu'$ (here $\phi_*\mu$ is the pushforward measure of $\mu$ under $\phi$). To ease notation, we will represent an equivalence class in $\XXbb_c$ by a single element of that equivalence class. 
	
First recall that if $(X, d)$ is a metric space, the {\bf Hausdorff distance} $d_H$ between two sets $A, A' \subset X$ is defined as
	\[
	d_H(A, A') = \max \{ \sup_{a \in A} d(a, A'), \sup_{a' \in A'} d(a', A) \}.
	\]
For $\epsilon>0$ and $A\subset X$ we also let $A^{\epsilon} = \{ x \in X: d(x,A) < \epsilon \}$ be the $\epsilon$-fattening of $A$ in $X$. If $\mu$ and $\nu$ are two measures on $X$, the {\bf Prohorov distance} between them is given by
	\[
	d_P(\mu, \nu) = \inf \{ \epsilon > 0: \mu(A) \leq \nu(A^{\epsilon}) + \epsilon \text{ and } \nu(A) \leq \mu(A^{\epsilon}) + \epsilon \text{ for any closed set } A \subset X \}.
	\]
	
	\begin{definition} \label{def:GHP} Let $(X,d,\mu)$ and $(X',d',\mu')$ be elements of $\XXbb_c$. The \textbf{Gromov-Hausdorff-Prohorov} {\rm (GHP)} distance between $(X,d,\mu)$ and $(X',d',\mu')$ is defined as
		\[
		\dGHP((X,d,\mu),(X',d',\mu')) = \inf \left\{d_H (\phi(X), \phi'(X')) \vee d_P(\phi_* \mu, \phi_*' \mu') \right\},
		\]
		where the infimum is taken over all isometric embeddings $\phi: X \rightarrow F$, $\phi': X' \rightarrow F$ into some common metric space $F$.
	\end{definition}

Recall that our aim in this paper is to prove \textit{distributional} convergence with respect to the GHP topology. Given an mm-space $(X,d,\mu)$ and a fixed $m \in \NN$ we define a measure $\nu_m((X,d,\mu))$ on $\RR^{m \choose 2}$ to be the law of the ${m \choose 2}$ pairwise distances between $m$ i.i.d.~points drawn according to $\mu$. Each law $\Pb$ on $\XXbb_c$ therefore defines random measures $\left(\nu_m\right)_{m\geq 2}$ and annealed measures $\left(\tilde{\nu}_m\right)_{m\geq 2}$ on $\mathbb{R}^{\binom{m}{2}}$, given by
	\begin{equation*}
	\tilde{\nu}_m(\Pb) := \int_{\XXbb_c} \nu_m((X,d,\mu))d\Pb.
	\end{equation*}

In \cite{ANS2021ghp} we rephrased a result of \cite[Theorem 6.1]{AthreyaLohrWinterGap} in the distributional setting to characterize GHP convergence in terms of convergence of the measures $(\tilde{\nu}_m)_{m \geq 2}$ and a volume condition. To state the version that we will use in this paper, given $c > 0$ and an mm-space $(X,d,\mu)$ we define
	\begin{align*}
	m_c((X,d,\mu)) &= \inf_{x \in X}\{\mu (B(x, c))\}
	\end{align*}
(cf \cite[Section 3]{AthreyaLohrWinterGap}).

In the proof of the next proposition we will also make reference to the (coarser) Gromov-Prohorov topology, which is defined as follows.

\begin{definition} \label{def:GP} Let $(X,d,\mu)$ and $(X',d',\mu')$ be elements of $\XXbb_c$. The \textbf{Gromov-Prohorov} {\rm (GP)} distance between $(X,d,\mu)$ and $(X',d',\mu')$ is defined as
		\[
		\dGP((X,d,\mu),(X',d',\mu')) = \inf \left\{ d_P(\phi_* \mu, \phi_*' \mu') \right\},
		\]
		where the infimum is taken over all isometric embeddings $\phi: X \rightarrow F$, $\phi': X' \rightarrow F$ into some common metric space $F$.
	\end{definition}

The key result is as follows.

\begin{proposition}\label{prop:GP plus LMB gives GHP}
Let $(X,d,\mu)$ be an element of $\XXbb_c$ with law $\Pb$ such that $\mu$ has full support almost surely. Let $((X_n,d_n,\mu_n))_{n\geq 1}$ be a sequence in $\XXbb_c$ with respective laws $(\Pb_n)_{n\geq 1}$ and suppose that:
		\begin{enumerate}[(a)]
			\item For all $m \geq 0$, $\tilde{\nu}_m(\Pb_n) \to \tilde{\nu}_m(\Pb)$ as $n \to \infty$.
			\item For any $c > 0$, the sequence $\left(m_c((X_n,d_n,\mu_n))^{-1}\right)_{n \geq 1}$ is tight.
		\end{enumerate}
		Then $(X_n, d_n, \mu_n) \overset{(d)}{\to} (X,d,\mu)$ with respect to the $\GHP$ topology.
\end{proposition}
\begin{proof}
First we show that part (a) and (b) together imply that $(X_n, d_n, \mu_n) \overset{(d)}{\longrightarrow} (X,d,\mu)$ with respect to the GP topology, by verifying the two conditions of \cite[Corollary 3.1]{GrevenPfaffWinterGwConvergence}. The second condition of \cite[Corollary 3.1]{GrevenPfaffWinterGwConvergence} is precisely (a). To verify the first condition we further use \cite[Theorem 3]{GrevenPfaffWinterGwConvergence} (recall that by Prohorov's Theorem the relative compactness of the measures is equivalent to their tightness) and verify conditions (i) and (ii) there (see also Proposition 8.1 in \cite{GrevenPfaffWinterGwConvergence}). Condition (i) is just saying that $\tilde{\nu}_2$ is a tight sequence of measures on $\RR$, which follows from (a). Lastly, (b) directly implies  condition (ii).

Therefore, by \cite[Theorem 6.5]{ANS2021ghp}, the spaces convergence with respect to the GHP topology.
\end{proof}

\section{Stick-breaking construction of trees}\label{sctn:stick breaking background}

Our first goal will be to prove condition (a) of \cref{prop:GP plus LMB gives GHP} which is equivalent to the following statement.

\begin{theorem}\label{thm:pr04}
Take $\gamma>0$ and $\delta>0$ and let $(G_n)_{n \geq 1}$ be a dense sequence of $\gamma$-expanders, where each $G_n$ has $n$ vertices and minimal degree at least $\delta n$. Denote by $d_{\T_n}$ the graph distance on $\T_n$ and by $(\T,d,\mu)$ the CRT.  Then there exists a sequence $(\beta_n)_{n \geq 1}$, satisfying $\sqrt{\delta} \leq \beta_n \leq 1$ for all $n \geq 1$, such that for any fixed $k \geq 1$, if $\{x_1,\ldots, x_k\}$ are uniformly chosen independent vertices of $G_n$, then the distances 
$$ \frac{d_{\T_n}(x_i,x_j)}{\beta_n \sqrt{n}} $$
converge jointly in distribution to the ${k \choose 2}$ distances in $\T$ between $k$ i.i.d.~points drawn according to $\mu$. 
\end{theorem}

To prove this theorem, we will use Aldous' stick-breaking construction of the CRT which is particularly well adapted to dealing with the pairwise distances between a set of $k$ uniform points. Our strategy will be to show that the first $k$ steps of Wilson's algorithm on $G_n$ closely approximate those of this stick-breaking process when $n$ is large. In this section we briefly recall the stick-breaking construction of the CRT and some of its key properties.

We start with a more general description of how one can construct a sequence of trees from sticks on the real line.

\begin{definition}(Stick-breaking construction of a tree sequence). \label{def:stick breaking}
Set $y_0=z_0=0$, and suppose that we have a sequence of points $y_1, y_2, \ldots \in [0, \infty)$ and $z_1, z_2, \ldots \in [0, \infty)$ such that $y_{i-1} < y_{i}$ and $z_i \leq y_{i}$ for all $i \geq 1$. Construct trees as follows. Start by taking the line segment $[y_0, y_1)$ at time $1$. This is $T^{(2)}$ (as it contains two marked points). We proceed inductively. At time $i \geq 2$, take the interval $[y_{i-1}, y_{i})$ and attach the base of the interval $[y_{i-1}, y_{i})$ to the point on $T^{(i)}$ corresponding to $z_{i-1}$. This gives a new tree with $i+1$ marked points (in bijection with the set $(y_j)_{j=0}^{i}$), which we call $T^{(i+1)}$.

Given two such sequences and any $k \geq 2$ we define $\SBk ((y_0, y_1, y_2, \ldots), (z_0, z_1, z_2, \ldots ))$ or equivalently $\SBk ((y_0, y_1, y_2, \ldots, y_{k-1}), (z_0, z_1, z_2, \ldots, z_{k-2} ))$ to be equal to the tree $\Tk$.
\end{definition}

In general, the sequence of trees constructed in this way above may not converge, but Aldous showed that by choosing the points in the right way, we can in fact construct the CRT via stick-breaking.

\begin{proposition}\cite[Process 3]{AldousCRTI}.\label{prop:CRT stick breaking}
Set $Y_0=Z_0=0$, let $(Y_1, Y_2, \ldots)$ denote the ordered set of points of a non-homogeneous Poisson process on $[0, \infty)$ with intensity $t \ dt$, and let $Z_i$ be chosen uniformly on the interval $[0,Y_{i})$ for each $i \geq 1$. Construct the sequence $(T^{(k)})_{k=2}^{\infty}$ as in \cref{def:stick breaking}. Then the closure of the limit of $T^{(k)}$ is equal in distribution to the CRT. Moreover, if one stops the process after $k-1$ steps, then the resulting tree $T^{(k)}$ has the same distribution as the subtree spanned by $k$ uniform points in the CRT, and the points corresponding to the set $(Y_i)_{i=0}^{k-1}$ can be identified with $k$ uniform points in the CRT.
\end{proposition}

In particular, the set of ${k \choose 2}$ pairwise distances between points corresponding $(Y_i)_{i=1}^k$ is equal in distribution to the set of ${k \choose 2}$ pairwise distances between $k$ uniform points in the CRT.

The following proposition will be important for the comparison with Wilson's algorithm later on. It can be verified by a direct computation.

\begin{proposition}\label{prop:CRT stick probs}
Define the sequence $(Y_1, Y_2, \ldots)$ as in \cref{prop:CRT stick breaking}. Then for any $k \geq 1$ and any $x \geq 0$,
\[
\prcond{Y_{k+1} - Y_k \geq x}{(Y_i)_{i=0}^k}{} = \exp\left\{-\frac{1}{2} \left((Y_k+x)^2 - Y_k^2\right)\right\}.
\]
\end{proposition}

The following lemma will also be useful.

\begin{lemma}\label{lem: size of stick}
There exists a function $f:[0, \infty) \times \NN \to [0,1]$ such that for every $k\in \NN$ we have that $\lim_{C\to\infty} f(C,k) \to 0$, and such that if $Y_k$ is as in \cref{prop:CRT stick breaking}, then
\begin{equation*}
    \pr{C^{-1} \leq Y_k \leq C} \geq 1 - f(C,k).
\end{equation*}
\end{lemma}

\begin{proposition}\label{prop:stick breaking close}
Let $(y_0, y_1, y_2, \ldots), (z_0, z_1, z_2, \ldots)$ and $(y_0', y_1', y_2', \ldots), (z_0', z_1', z_2', \ldots)$ be the inputs to two separate stick-breaking processes as defined in Definition \ref{def:stick breaking}. Fix any $k \geq 1$ and let $T^{(k+1)}$ and $T^{(k+1)'}$ be the trees formed after $k$ steps of the processes. Let $d$ and $d'$ denote distances on $T^{(k+1)}$ and $T^{(k+1)'}$.

Fix some $\epsilon>0$ and suppose that the following holds.
\begin{enumerate}[(i)]
    \item $|y_i - y_i'| \leq \epsilon$ for all $i \leq k$ and $|z_i - z_i'| \leq \epsilon$ for all $i \leq k-1$,
    \item $|z_i - y_j| \geq 3\epsilon$ for all $i \leq k-1,j \leq k$.
\end{enumerate}
Then, for all $0 \leq i,j \leq k$, it holds that
\[
|d(y_i, y_j) - d'(y_i', y_j')| \leq 2k\epsilon.
\]
\end{proposition}
\begin{proof}
When conditions $(i)$ and $(ii)$ hold, we have for all $i \leq k-1,j \leq k$ that $y_j \leq z_i \leq y_{j+1}$ if and only if $y_j' \leq z_i' \leq y_{j+1}'$. We claim that this implies that $|d(y_i, y_j) - d'(y_i', y_j')| \leq 2k\epsilon$ for all $i, j \leq k+1$. Indeed, it follows by construction that $d(y_i, y_j)$ is the sum of lengths of at most $k$ branch segments in $T^{(k+1)}$, and all of their lengths can be written in the form $|y_j - y_{j-1}|, |z_j - y_{\ell}|$ or $|z_j - z_{\ell}|$. Moreover, by construction, when the conditions $(i)$ and $(ii)$ hold, $d'(y_i', y_j')$ can be written as the same sum but replacing each $z_j$ with $z_j'$ and replacing each $y_j$ with $y_j'$. It therefore follows from the triangle inequality that $|d(y_i, y_j) - d'(y_i', y_j')| \leq 2k\epsilon$.
\end{proof}

\section{Random walk properties}\label{sctn:RW estimates}
In this section we prove some results on random walk hitting probabilities and capacity, which we will later transfer to segments of LERW using the Laplacian random walk representation of \cref{sctn:Laplacian RW}.

Throughout the section we fix a small $\neweps \in (0, \frac{1}{32})$ and for $n \geq 1$ we set $M_n = n^{\neweps}$. In what follows we will simply write $M$ instead of $M_n$.

\textbf{Notational remark.} For the statements in this section, we will take a sequence of graphs satisfying the assumptions of \cref{thm:main2} which is therefore associated with two positive constants $\gamma>0$ and $\delta>0$. In this section we will treat these constants as fixed, and therefore $o(\cdot)$ and $O(\cdot)$ quantities may also depend on $\gamma$ and $\delta$.

\subsection{Hitting probabilities}
We start with some results on hitting probabilities. Let $X$ be a (non-lazy) random walk on $G_n$ for some $n \geq 1$. For a set $A \subset V(G_n)$, we define 
\[
\tau_A = \inf\{t \geq 0: X_t \in A\}.
\]
The main lemma is the following.

 \begin{lemma}\label{lem:hitting ratio precise}
Take $\gamma>0$ and $\delta>0$ and let $(G_n)_{n \geq 1}$ be a dense sequence of $\gamma$-expanders, where each     $G_n$ has $n$ vertices and minimal degree at least $\delta n$. Take $\neweps$ and $M$ as defined at the start of \cref{sctn:RW estimates}. Then there exists a sequence $(\eta_n)_{n \geq 1}$ with $\eta_n \to 0$, depending only on $\delta$ and $\gamma$, such that for any disjoint $A, B \subset G_n$ satisfying $|A| + |B| \leq \frac{\delta^3}{2}n^{\frac{1}{2}+2\neweps}$:
 \[
 \left|\prstart{\tau_A < \tau_B}{\pi} -\frac{\Capp_M(A)}{\Capp_M(A)+\Capp_M(B)} \right| \leq \frac{\Capp_M(A)\eta_n}{\Capp_M(A)+\Capp_M(B)}.
 \]
\end{lemma}
 \begin{proof}
Let $(X_i)_{i=1}^M$ be a random walk of length $M$. Then, by Bayes' formula, \cref{cor:closeness UB deter} and the lower bound in \cref{lem:cap union bound},
\begin{align*}
\prcond{\tau_A <M}{\tau_A \wedge \tau_B < M}{\pi} &= \frac{\Capp_M(A)}{\Capp_M(A) + \Capp_M(B) - \prstart{\tau_A \vee \tau_B <M}{\pi}} \\
&= \frac{\Capp_M(A)}{\Capp_M(A) + \Capp_M(B)} \left( 1+O\left(\frac{\delta^{-3}M|B|}{n}\right)\right). \\
\prcond{\tau_A \vee \tau_B < M}{\tau_A \wedge \tau_B < M}{\pi} &= \frac{\prstart{\tau_A \vee \tau_B <M}{\pi}}{\Capp_M(A) + \Capp_M(B) - \prstart{\tau_A \vee \tau_B <M}{\pi}} \\
&\leq \prcond{\tau_A <M}{\tau_A \wedge \tau_B < M}{\pi} O\left( \frac{\delta^{-3}M|B|}{n} \right).
\end{align*}
Therefore, combining these and applying \cref{lem:cap union bound}:
\begin{align*}
\prcond{\tau_A < \tau_B}{\tau_A \wedge \tau_B < M}{\pi} &= \prcond{\tau_A <M}{\tau_A \wedge \tau_B < M}{\pi} +O( \prcond{\tau_A \vee \tau_B <M}{\tau_A \wedge \tau_B < M}{\pi}) \\
&= \frac{\Capp_M(A)}{\Capp_M(A) + \Capp_M(B)}\left( 1 +O\left(\frac{\delta^{-3}M|B|}{n}\right)\right).
\end{align*}
It similarly follows from \cref{claim:capacity different start point} and the lower bound in \cref{lem:cap union bound} that uniformly over all $u \in G_n \setminus (A \cup B)$,
\begin{align*}
\prcond{\tau_A < \tau_B}{\tau_A \wedge \tau_B < M}{u} &= \frac{\Capp_M(A)}{\Capp_M(A) + \Capp_M(B)}\left( 1 +O\left(\frac{\delta^{-3}M|B|}{n} + \frac{\tmix \cdot \log n}{\delta^2 M}\right)\right).
\end{align*}

Now we decompose time into intervals of length $M$. For each $i \geq 1$, define the interval $A_i$ by
\[
A_i = [iM, (i+1)M].
\]
We then have that, using \cref{cor:closeness UB deter}:
\begin{align*}
 \prstart{\tau_A < \tau_B}{\pi} &\geq \sum_{i =0}^{\infty} \prcond{\tau_A < \tau_B}{\tau_{A \cup B} \in A_i}{\pi} \prstart{\tau_{A \cup B} \in A_i}{\pi} \\
 &\geq \sum_{i =0}^{\infty} \inf_{u \in G_n \setminus (A \cup B)} \prcond{\tau_A < \tau_B}{\tau_{A \cup B} \in A_0}{u} \prstart{\tau_{A \cup B} \in A_i}{\pi} \\
 &\geq \frac{\Capp_M(A)}{\Capp_M(A) + \Capp_M(B)}\left( 1 +O\left( \frac{\delta^{-3}|B|M}{n} + \frac{\tmix \cdot \log n}{\delta^2 M}\right)\right) 
\end{align*}
We deduce that, uniformly over all permitted $A$ and $B$,
\begin{equation}
    \prstart{\tau_A < \tau_B}{\pi} \geq \frac{\Capp_M(A)}{\Capp_M(A) + \Capp_M(B)}(1-o_{\delta, \gamma}(1)),
\end{equation}
where the $o_{\delta, \gamma}(1)$ term is uniform over all $A$ and $B$ but may depend on $\delta$ and $\gamma$. Similarly, for an upper bound on  $\prstart{\tau_A < \tau_B}{\pi}$ we simply exchange the roles of $A$ and $B$.
We deduce that, uniformly over all permitted $A$ and $B$,
\begin{equation}
    \prstart{\tau_A < \tau_B}{\pi} = \frac{\Capp_M(A)}{\Capp_M(A) + \Capp_M(B)}(1-o_{\delta, \gamma}(1)).
\end{equation}
\end{proof}

We will also need the following minor adaptation.

\begin{lemma}\label{lem:hitting ratio fixp}
Take $\gamma>0$ and $\delta>0$ and let $(G_n)_{n \geq 1}$ be a dense sequence of $\gamma$-expanders, where each $G_n$ has $n$ vertices and minimal degree at least $\delta n$. Take $\neweps$ and $M$ as defined at the start of \cref{sctn:RW estimates}. Then, for any disjoint $A, B \subset G_n$ satisfying $|A| + |B| \leq \frac{\delta^3}{2}n^{\frac{1}{2}+2\neweps}$, every $u\in G_n \setminus (A \cup B)$ and every $v\in G_n \setminus A$ we have that
\begin{equation*}
 \prstart{\tau_A < \tau_B}{u} = \prstart{\tau_A < \tau_B}{\pi}(1+o(n^{3\neweps-1/2}\tplus)) \ \ \text{ and } \ \ \prstart{\tau_A < \tau_B^+}{v} = \prstart{\tau_A < \tau_B}{\pi}(1+o(n^{3\neweps-1/2}\tplus)) .
\end{equation*}
\end{lemma}
\begin{proof}
We start by proving the first statement for a lazy random walk, since this is equivalent, and we denote such a lazy random walk by $X$. Throughout this proof, we will also use the following notation. For a set $C \subset G_n$ and some time $t \geq 0$ we write $\tau(C,t)$ for the first time $s$ strictly larger than $t$ such that $X_s \in C$. Furthermore, write $\tmix^+$ for $\log_2^2(n)\tmix$ so that by \eqref{eq:tvDecay} we have that for every $u\in G_n$,
\begin{equation*}
    	d_{\textrm{TV}}(p_{\tmix^+}(u,\cdot),\pi(\cdot)) \leq n^{-\log(2)\log(n)}.
\end{equation*}

Now let $u\in G_n \setminus A$. We start with a lower bound on $\prstart{\tau_A < \tau_B}{u}$. We have that
\begin{align}\label{eqn:return time LB comp}
\begin{split}
\prstart{\tau_A < \tau_B}{u} &\geq  \prstart{\tmix^+ < \tau_A < \tau_B}{u} \\&\geq \prstart{\tau(A,\tmix^+) < \tau(B,\tmix^+)}{u} - \prstart{\tau_{A\cup B} < \tmix^+ < \tau(A,\tmix^+) < \tau(B,\tmix^+)}{u} 
\end{split}
 \end{align}
 Note that by \eqref{eq:tvDecay}, the first term can be lower bounded by $\prstart{\tau_A < \tau_B}{\pi} - n^{-\log(2)\log(n)}$. For the second term, let us upper bound the probability of the event $\{\tau_{A\cup B} < \tmix^+ < \tau(A,\tmix) < \tau(B,\tmix^+)\}$. Using a union bound we obtain
 \begin{align*}
     &\prstart{\tau_{A\cup B} < \tmix^+ < \tau(A,\tmix^+) < \tau(B,\tmix^+)}{u} \\&\leq \prstart{\tau_{A\cup B} < \tmix^+ < \tau(A,\tmix^+) < 2\tmix^+}{u} + \prstart{\tau_{A\cup B} < \tmix^+ \pand \tau(A,2\tmix^+) < \tau(B,2\tmix^+)}{u}.
     \\&\leq \frac{|A|\tmix^+}{\delta n} 
     + \frac{(|A| + |B|)\tmix^+}{\delta n}\cdot(\prstart{\tau_A < \tau_B}{\pi} + n^{-\log(2)\log(n)}). 
 \end{align*}
Note that, by \cref{lem:cap union bound} and \cref{lem:hitting ratio precise} we have that
\begin{equation*}
    \frac{|A|\tplus}{\prstart{\tau_A<\tau_B}{\pi}\delta n} \leq \frac{|A|\tplus}{\delta n}\cdot \frac{2(|A|+|B|)}{\delta^2|A|} \leq \frac{2(|A|+|B|)\tplus}{\delta^3 n},
\end{equation*}
so that, by \cref{claim:mixing time expander}
\begin{equation}\label{eq: error on difference from pi}
    \frac{|A|\tplus}{\delta n} = \prstart{\tau_A < \tau_B}{\pi}\cdot O_{\gamma}(n^{2\neweps-1/2}\tplus) = \prstart{\tau_A < \tau_B}{\pi}\cdot o_{\gamma}(n^{3\neweps-1/2}\tplus) 
\end{equation}
Substituting everything back into \eqref{eqn:return time LB comp}, we therefore deduce that
\begin{equation*}
    \prstart{\tau_A < \tau_B}{u} \geq \prstart{\tau_A<\tau_B}{\pi}(1+o(n^{3\neweps-1/2}\tplus)).
\end{equation*}
For an upper bound on $\prstart{\tau_A < \tau_B}{u}$, we simply write \begin{equation*}
    \prstart{\tau_A < \tau_B}{u} \leq \prstart{\tau_A < \tmix^+}{u} + \prstart{\tmix^+ < \tau(A,\tmix^+) < \tau(B,\tmix^+)}{u}\leq \frac{|A|\tmix^+}{\delta n} + \prstart{\tau_A < \tau_B}{\pi} + n^{-\log(2)\log(n)}.
\end{equation*}
Using \eqref{eq: error on difference from pi} again we obtain that
\begin{equation*}
    \prstart{\tau_A<\tau_B}{u} = \prstart{\tau_A<\tau_B}{\pi}(1+o(n^{3\neweps-1/2}\tplus)).
\end{equation*}
For the second statement, it is again enough to prove it for the lazy random walk, replacing $\tau_B^+$ with the first hitting time of $B$ after making at least one non-lazy step using the exact same proof.
 \end{proof}

\subsection{Capacity}
Here we prove some similar properties for the capacity and closeness of a random walk.

In this section we can also introduce the sequence $(\alpha_n)_{n \geq 1}$ appearing in \cref{thm:main2}. Given the graph sequence $(G_n)_{n \geq 1}$, take $M=n^{\neweps}$ as defined at the start of \cref{sctn:RW estimates}, let $X$ be a random walk on $G_n$, and for each $n \geq 1$ set
\[
\alpha_n = \frac{n\estart{\Capp_M (X[0, n^{\neweps/2}))}{\pi}}{Mn^{\neweps/2}}.
\]

\begin{proposition}\label{prop:RW cap}
Take $\gamma>0$ and $\delta>0$ and let $(G_n)_{n \geq 1}$ be a dense sequence of $\gamma$-expanders, where each $G_n$ has $n$ vertices and minimal degree at least $\delta n$. Let $u \in G_n$ and let $(X_i)_{i \geq 0}$ denote a random walk on $G_n$ started at $u$. Take $M=n^{\neweps}$ as defined at the start of \cref{sctn:RW estimates}. Then for all sufficiently large $n$,
\[
\pr{\left|\Capp_M (X[0, M)) - \frac{\alpha_n M^2}{n}\right| \geq  \frac{\alpha_n M^2}{n}n^{-\neweps/16}} \leq \frac{2M^2}{\delta n}.
\]
\end{proposition}

\begin{proof} The proof is a simplified version of that of \cite[Lemma 5.3]{PeresRevelleUSTCRT}. First recall from \cref{lem:hitting ratio fixp} that $\tmix^+=\log_2^2(n)\tmix$. Let $(T_j)_{j=1}^{n^{\neweps/2}}$ be a sequence of i.i.d random variables with distribution $\bin(\tmix^+, 1/2)$. Then, for all $1\leq j \leq n^{\neweps/2}$ let
\begin{equation*}
    B_j = [(j-1)n^{\neweps/2} + T_j, jn^{\neweps/2} - \tmix^+ + T_j].
\end{equation*}
Note that by \eqref{eq:tvDecay} we have that for all $j\leq n^{\neweps/2}$, given $X[0,jn^{\neweps/2}]$, the starting point of $B_{j+1}$ is nearly stationary. 
Also let $(\Xindj)_{j=1}^{n^{\neweps/2}}$ denote a sequence of independent random walk segments each of length $n^{\neweps/2} - \tplus$, and each started from stationarity. Note that, by \eqref{eq:tvDecay}, the segments $(X_{B_j})_{j=1}^{n^{\neweps/2}}$ can be coupled with the segments $(\Xindj)_{j=1}^{n^{\neweps/2}}$ so that the segments coincide for all $j \leq n^{\neweps/2}$ with probability at least
\begin{equation}\label{eqn:prob coupling}
1-n^{\neweps/2}n^{-\log_2(n)}.
\end{equation}

Note that the segments $(\Xindj)_j$ are i.i.d. and, by definition,
\begin{align}\label{eqn:cap segment error}
\E{\Capp_M (\Xindj)} &= \E{\Capp_M (\Xindj_{[0, n^{{\neweps}/{2}}) })} + O\left(\frac{M\tplus}{\delta n}\right) = \frac{\alpha_n M n^{{\neweps}/{2}}}{n} \left(1 + O\left(\frac{\tplus}{\delta n^{\neweps/2}}\right) \right).
\end{align}
Moreover, by a union bound, we also have the deterministic bound
\begin{align}\label{eqn:cap deterministic UB}
\Capp_M (\Xindj) \leq M \pi (\Xindj ) \leq \frac{M n^{\neweps/2}}{\delta n}. 
\end{align}
It therefore follows from a Hoeffding bound \cite[Theorem 1]{hoeffding1994probability} that there exist $C < \infty, c>0$ such that for any $t>0$,
\begin{align*}
\pr{\left|\sum_{j=1}^{n^{\neweps/2}} \Capp_M (\Xindj) - n^{\neweps/2}\E{\Capp_M(\Xindone)}\right| \geq \frac{\alpha_n M^2 \tplus}{2n^{1+\neweps/8}}} &\leq
2 \exp \left(-2n^{\neweps/2}\left(\frac{\frac{\alpha_n M^2 \tplus}{2n^{1+5\neweps/8}}}{\frac{Mn^{\neweps/2}}{\delta n}}\right)^{2} \right) 
\\&= 2 \exp \left(-n^{\neweps/4}(\tplus)^2\frac{\alpha_n^2\delta^2}{2}\right).
\end{align*}
In particular, since it follows from \eqref{eqn:cap segment error} that
\[
\left| n^{\neweps/2}\E{\Capp_M(\Xindone)} - \frac{\alpha_n M^2}{n} \right| \leq O \left( \frac{\alpha_n M n^{{\neweps}/{2}}}{n} \frac{\tplus}{\delta} \right) \ll \frac{\alpha_n M^2 \tplus}{n^{1+\neweps/8}},
\]
we deduce that
\begin{align}\label{eqn:Hoeffding}
    \pr{\left|\sum_{j=1}^{n^{\neweps/2}} \Capp_M (\Xindj) - \frac{\alpha_n M^2}{n}\right| \geq \frac{\alpha_n M^2 \tplus}{n^{1+\neweps/8}}} \leq 2 \exp \left(-n^{\neweps/4}(\tplus)^2\frac{\alpha_n^2\delta^2}{2}\right).
\end{align}
We would like to approximate the capacity of the whole segment $X[iM, (i+1)M)$ by the sum of the capacities of the smaller segments, but this is potentially a slight overestimate, since we are double-counting random walk trajectories that hit more than one smaller segment. To account for this, we use the concept of closeness defined in \cref{sctn:capacity defs}. For each $J \leq n^{\neweps/2}$, note that conditionally on $(\Xindj)_{j \leq J}$ all being disjoint, which happens with probability at least
\begin{equation}\label{eqn:prob disjoint}
    1 -  \frac{M^2}{\delta n},
\end{equation}
we have by \cref{cor:closeness UB deter} that
\begin{align*}
\close_M(\Xindjj, \cup_{j<J} (\Xindj)_j) \leq   \frac{2M^3n^{\neweps/2}}{\delta^2n^2} .
\end{align*}
Equally, approximating $\Capp_M (X[0, M))$ by $\sum_{j=1}^{n^{\frac{\neweps}{2}}} \Capp_M (X_{B_j})$ might be undercounting slightly, since there is also a contribution to the capacity from the set $X[0,M) \setminus \left(\cup_{j<n^{\neweps/2}} X_{B_j}\right)$.

We now combine the above estimates as follows. Note that, on the event $X_{B_j} = \Xindj$ for all $j \leq J$, we have (also using \eqref{eqn:cap deterministic UB} and a union bound) that for all sufficiently large $n$:
\begin{align}\label{eqn:subtract closeness}
\begin{split}
\left|\Capp_M (X[0, M)) - \sum_{j=1}^{n^{\frac{\neweps}{2}}} \Capp_M (\Xindj) \right| &\leq \sum_{J=1}^{n^{\neweps/2}} \close_M(X_{B_j}, \cup_{j<J} X_{B_j}) + \pr{\tau_{X[0, M) \setminus \left(\cup_{j<n^{\neweps/2}} X_{B_j} \right)} \leq M} \\
&\leq \frac{2M^4}{\delta^2n^{2}} + \frac{Mn^{\neweps/2}\tplus}{\delta n} \leq n^{-\neweps/8}\frac{\alpha_n M^2 \tplus}{n^{1+\neweps/8}}.
\end{split}
\end{align}
Therefore, combining with the estimates of \eqref{eqn:prob coupling}, \eqref{eqn:Hoeffding}, \eqref{eqn:prob disjoint} in a union bound, applying \eqref{eqn:subtract closeness} and using that $\tplus \ll n^{\neweps/32}$, we see that with probability at least $1 - \frac{2M^2}{\delta n}$ we have that
\begin{align*}
&\left|\Capp_M (X[0, M)) - \frac{\alpha_n M^2}{n} \right| \leq o \left(
\frac{\alpha_n M^2}{n} \cdot n^{-\neweps/16}\right). \qedhere
\end{align*}
\end{proof}

\section{Laplacian random walk representation and Wilson's algorithm}\label{sctn:Wilson analysis}

Throughout all of this section, we let $(G_n)_{n \geq 1}$ be a sequence of graphs satisfying the assumptions of \cref{thm:main2} with parameters $\delta>0$ and $\gamma>0$. By \cref{claim:mixing time expander}, this implies that $\tmix = O(\log n)$. For each $n, k \geq 1$, $\Tkkn$ will denote the tree obtained after running Wilson's algorithm on $G_n$ on the vertex set $(v_1, \ldots, v_{k-1})$. Given such a sequence $(G_n)_{n \geq 0}$, we set
\begin{equation}\label{eq: def of beta n}
\alpha_n = \frac{n\estart{\Capp_M (X[0, n^{\neweps/2}))}{\pi}}{Mn^{\neweps/2}}, \qquad \beta_n = \frac{1}{\sqrt{\alpha_n}},
\end{equation}
where $X$ is a random walk on $G_n$ and $\neweps$ is as defined at the start of \cref{sctn:RW estimates}. Lastly, if $A \subset G_n$, we will use the notation $\tau_A$ to denote the hitting time of $A$ for a random walk on $G_n$.

The goal of this section is to prove the forthcoming \cref{cor:main stick breaking ingredient} for such a sequence of graphs, for which we will need the following definition.

\begin{definition}\label{def:good tree}
We say that a subgraph $T \subset G_n$ is good if
\begin{enumerate}
\item $T$ is a tree.
    \item $|T| \leq n^{1/2+\neweps}$.
    \item For every open connected subset $A\subset T$ with $|A| \geq n^{3\neweps}$ we have that $|\Capp_M(A) - \alpha_n M |A| / n| \leq \frac{\alpha_n M|A|}{n}\cdot n^{-\neweps/16}$.
\end{enumerate}
\end{definition}

Note that $T^{(1)}_n$, the tree consisting of the first single vertex, is trivially good.

\begin{proposition}
\label{cor:main stick breaking ingredient}
Take any good subgraph $T \subset G_n$. Take any $u \in V(G_n) \setminus T$ and let $Y$ be a LERW started at $u$ and terminated when it hits $T$.
\begin{enumerate}[(1)]
    \item Take any $C \in (0, \infty)$ and any $B\in (0,\infty)$. Suppose additionally that $|T| \leq B\sqrt{n}$. Then
    \begin{align*}
        \prstart{Y[0, C \beta_ns\sqrt{n}] \cap T = \emptyset}{u} = \exp\left\{- \frac{(C + |T|/\beta_n \sqrt{n})^2-(|T|/\beta_n\sqrt{n})^2}{2} \right\} ( 1 + o_B(1)).
    \end{align*}
     \item Let $H_T$ be the time that $Y$ hits $T$. Then for any connected $A \subset T$ with $|A| \geq n^{3\neweps}$ and for all $n^{2\neweps}/M \leq i \leq n^{1/2+\neweps}/M$,
\begin{align*}
\prcond{Y_{H_{T}} \in A}{H_{T} \in [iM, (i+1)M)}{u} = \frac{|A|}{|T|}(1 + o(1)),
\end{align*}
where the $o(1)$ is uniform over all $n^{2\neweps}/M \leq i\leq n^{1/2+\neweps}/M$.
\item For any $k \geq 1$ and $B\in(0,\infty)$, if $|T|\leq B\sqrt{n}$ then
\begin{equation*}
    \prcond{\Tkn \text{is good}}{\Tkkn=T}{} = 1-o_B(1).
\end{equation*}
\end{enumerate}
\end{proposition}
\cref{cor:main stick breaking ingredient} will allow us to compare Wilson's algorithm with the CRT stick-breaking process in \cref{sctn:Wilson is stickbreaking} and prove \cref{thm:main2}.

\subsection{Comparison of path probabilities}

We will use the Laplacian random walk representation of LERW outlined in \cref{sctn:Laplacian RW} to compare the probability of different LERW trajectories. We first fix some $n \geq 1$ and we run Wilson's algorithm on $G_n$. Let $\{v_1, \ldots, v_n\}$ denote the ordering of the vertices for this process.  

Now fix some $k \geq 2$, suppose we have run $k-2$ steps of Wilson's algorithm to form a tree spanned by the vertices $(v_1, \ldots, v_{k-1})$, which we denote by $\Tkkn$. Let $X$ denote a random walk on $G_n$, killed when it hits $\Tkkn$, and let $(Y_m)_{m \geq 0}$ denote its loop erasure. Recall from \cref{sctn:Laplacian RW} that, if $u_0, u_1, \ldots, u_H$ is a simple path in $G_n$, where $\{u_0, u_1, \ldots, u_{H-1}\}$ is disjoint from $\bigcup_{m=0}^{L} \{Y_m\} \cup \Tkkn$ and $u_0 = Y_L$, then
\begin{align}\label{eqn:LERW path factorisation}
\begin{split}
\prcond{ (Y_m)_{m=L+1}^{L+H} = (u_m)_{m=1}^{H}}{(Y_m)_{m=0}^{L}}{} &= \prstart{ (X_m)_{m=1}^{H} = (u_m)_{m=1}^{H}}{u_0} C((Y_m)_{m=0}^{L}, \Tkkn, (u_m)_{m=1}^{H})),
\end{split}
\end{align}
where 
\begin{equation}\label{eqn:C constant expression}
C((Y_m)_{m=0}^{L}, \Tkkn, (u_m)_{m=1}^{H}) = \prod_{h=1}^{H} \frac{\prstart{\tau_{\Tkkn} < \tau_{\{Y_m\}_{m=0}^{L} \cup \ \{u_m\}_{m=0}^{h-1}   }}{u_h}}{\prstart{\tau_{\Tkkn} < \tau^+_{ \{Y_m\}_{m=0}^{L} \cup \  \{u_m\}_{m=0}^{h-1}}}{u_{h-1}}}.
\end{equation} 

At some points in this section, we will condition on an event of the form $\Tkkn=T$ and $X$ and $Y$ will respectively denote a random walk and a LERW, both killed when they hit $T$. For notational convenience, if $H_T$ is the time that $Y$ hits $T$, we set $Y_m = Y_{H_T}$ for all $m \geq H_T$.

\begin{remark}
To prove \cref{thm:pr04}, we should choose the vertices $\{v_1, \ldots, v_k\}$ uniformly on $G_n$. In fact we will prove a result that holds for any choice of distinct $\{v_1, \ldots, v_k\}$.
\end{remark}

The strategy to prove \cref{cor:main stick breaking ingredient} will be roughly as follows. Firstly, \cref{lem:dont hit too soon} enables us to control the behavior of a first small segment of $Y$. This will enable us to give tight estimates for the constant $C$ defined by \eqref{eqn:C constant expression}, which we do in \cref{lem:hitting lem LB rough same event}. In \cref{cor:capacity tail} we combine this with \cref{prop:RW cap} and \eqref{eqn:LERW path factorisation} to tightly control the capacity of LERW segments. In \cref{lem:hitting without intersections} we estimate the random walk hitting measure on a good tree $T$, and in \cref{cor:hit A in next interval} we combine this with the estimates on the constant $C$ to plug into \eqref{eqn:LERW path factorisation} and obtain analogous estimates for LERW.

\begin{lemma}\label{lem:dont hit too soon}
There exists $N<\infty$ such that for all $n \geq N$ the following holds. Let $T \subset G_n$ be a subgraph with $|T| \leq n^{\frac{1}{2}+\neweps}$. Take any $u \in V(G_n) \setminus T$ and let $Y$ be a LERW started at $u$ and terminated when it hits $T$. Then
\[
\prstart{|Y| \leq n^{2\neweps}}{u} \leq \frac{4n^{2\neweps}(n^{2\neweps} + |T|)}{\delta^3 n}.
\]
\end{lemma}
\begin{proof}
For each $K < n^{2\neweps}$, we will use the Laplacian random walk representation to bound the probability that $\pr{|Y| = K+1 \mid |Y| > K}$. To this end, we have by \eqref{eqn:LERW path factorisation} for every $v \in T$ and every simple path $\varphi$ of length $K$ with $\varphi \cap T = \emptyset$ that
\begin{equation*}
    \prcond{Y_{K+1} = v}{(Y_m)_{m=0}^K = \varphi}{u} \leq \frac{1}{\delta n} \cdot \frac{1}{
    \prstart{\tau_{T} < \tau^+_{\varphi}}{Y_K}{}
    }.
\end{equation*}
By \cref{lem:hitting ratio precise} and \cref{lem:hitting ratio fixp}, the second term can be bounded by estimating the capacities of $T$ and of $\varphi$. We claim that up to some constants depending on the minimal degree, they both can be estimated by their sizes. Indeed, for any set of $A$ size less than $n^{1/2+\neweps}$, we have that $\frac{\delta M |A|}{2n} \leq \Capp_M(A) \leq \frac{|A|M}{\delta n}$ when $n$ is large enough by \cref{lem:cap union bound}.

Therefore, as $\varphi$ is of size $K < n^{2\neweps} < n^{1/2+\neweps}$ and $T$ is of size at most $n^{1/2+\neweps}$ we have that
\begin{equation*}
    \frac{\Capp_M(T)}{\Capp_M(T) + \Capp_M(\varphi)} \geq \frac{\delta^2|T|}{2(n^{2\neweps} + |T|)}.
\end{equation*}
Therefore, by \cref{lem:hitting ratio precise}, \cref{lem:hitting ratio fixp} and summing over all $v \in T$ we have for all sufficiently large $n$ that
\begin{equation*}
    \prcond{Y_{K+1} \in T}{(Y_m)_{m=0}^K = \varphi}{u} \leq \frac{4(n^{2\neweps} + |T|)}{\delta^3 n}.
\end{equation*}
By a union bound, we can thus conclude that $\prstart{|Y| \leq n^{2\neweps}}{u} \leq \frac{4n^{2\neweps}(n^{2\neweps} + |T|)}{\delta^3 n}$ as required.
\end{proof}

We will also need the following result to control the constant $C$ defined by \eqref{eqn:C constant expression}. The reader should have in mind that we will eventually apply the result with $T=\Tkkn$.

\begin{lemma}\label{lem:hitting lem LB rough same event}
Take $i \leq \frac{n^{1/2+\neweps}}{M}$. Then, for all simple paths $\{Y_m\}_{m=0}^{iM}$, all simple paths $\{u_m\}_{m=0}^H$ such that $u_0=Y_{iM}$ and $H \leq M$, and all connected subgraphs $T \subset G_n$ such that $\{u_m\}_{m=1}^{H-1}, \{Y_m\}_{m=0}^{iM}$ and $T$ are disjoint and such that $|T| \leq n^{1/2+\neweps}$ we have the following
\begin{enumerate}[(a)]
    \item If $u_H \notin T$, then
\begin{equation*}
\left|C((Y_m)_{m=0}^{iM}, T, (u_m)_{m=1}^{H}) - 1 \right| = o(n^{5\neweps-1/2}).
\end{equation*}
\item
If $u_H \in T$ and $i \geq n^{2\neweps}/M$, then
\begin{equation*}
    C((Y_m)_{m=0}^{iM}, T, (u_m)_{m=1}^{H}) = \frac{[\Capp_M(T) + \Capp((Y_m)_{m=0}^{iM})](1+o(1))}{\Capp_M(T)}.
\end{equation*}
\end{enumerate}
\end{lemma}
\begin{proof}
Fix some $1\leq h \leq H$. In case (a), in order to bound a term appearing in the product in \eqref{eqn:C constant expression}, we would like to compare the probabilities
\begin{equation}\label{eq:compare two probs}
    \prstart{\tau_{T} < \tau_{\{Y_m\}_{m=0}^{iM} \cup \ \{u_m\}_{m=0}^{h-1}}}{u_h} \ \pand \ \prstart{\tau_{T} < \tau^+_{\{Y_m\}_{m=0}^{iM} \cup \ \{u_m\}_{m=0}^{h-1} }}{u_{h-1}}.
\end{equation}
By \cref{lem:hitting ratio fixp} and by the triangle inequality, we have that
\begin{align*}
    &\left|\prstart{\tau_{T} < \tau_{\{Y_m\}_{m=0}^{iM} \cup \ \{u_m\}_{m=0}^{h-1}}}{u_h} - \prstart{\tau_{T} < \tau^+_{\{Y_m\}_{m=0}^{iM} \cup \ \{u_m\}_{m=0}^{h-1}}}{u_{h-1}}\right| \\&\leq \prstart{\tau_{T} < \tau_{\{Y_m\}_{m=0}^{iM} \cup \ \{u_m\}_{m=0}^{h-1}}}{u_h}(o(n^{3\neweps-1/2}\tplus)).
\end{align*}
In other words,
\begin{equation*}
    \frac{\prstart{\tau_{T} < \tau_{\{Y_m\}_{m=0}^{iM} \cup \ \{u_m\}_{m=0}^{h-1}}}{u_h}}{\prstart{\tau_{T} < \tau^+_{\{Y_m\}_{m=0}^{iM} \cup \ \{u_m\}_{m=0}^{h-1}}}{u_{h-1}}} = 1+o(n^{3\neweps-1/2}\tplus).
\end{equation*}
Hence we have that the product in \eqref{eqn:C constant expression} is bounded by
\begin{equation*}
    (1+o(n^{3\neweps-1/2}\tplus))^{n^{\neweps}} = 1 + o(n^{4\neweps - 1/2}\tplus).
\end{equation*}
To conclude, we use that $\tmix = o(n^{\neweps})$ as $n \to \infty$ by \cref{claim:mixing time expander}.

For (b), if $u_H \in T$, then $C((Y_m)_{m=0}^{iM}, T, (u_m)_{m=1}^{H})$ is instead equal to
\[
\frac{1}{\prstart{\tau_{T} < \tau^+_{\{Y_m\}_{m=0}^{iM} \cup \ \{u_m\}_{m=0}^{h-1}}}{u_{H-1}}} \prod_{h=1}^{H-1} \frac{\prstart{\tau_{T} < \tau_{\{Y_m\}_{m=0}^{iM} \cup \ \{u_m\}_{m=0}^{h-1}}}{u_h}}{\prstart{\tau_{T} < \tau^+_{\{Y_m\}_{m=0}^{iM} \cup \ \{u_m\}_{m=0}^{h-1}}}{u_{h-1}}},
\]
where the product is equal to $C((Y_m)_{m=0}^{iM}, T, (u_m)_{m=1}^{H-1})$ and is therefore $1+o(1)$ by (a). Then note that by \cref{lem:hitting ratio precise} and \cref{lem:hitting ratio fixp}, we have that 
\begin{align*}
\prstart{\tau_{T} < \tau^+_{\{Y_m\}_{m=0}^{iM} \cup \ \{u_m\}_{m=0}^{h-1}}}{u_{H-1}} &= \frac{\Capp_M(T)(1+o(1))}{\Capp_M(T) + \Capp_M((Y_m)_{m=0}^{iM} \cup \ \{u_m\}_{m=0}^{h-1})} \\&= \frac{\Capp_M(T)(1+o(1))}{\Capp_M(T) + \Capp_M((Y_m)_{m=0}^{iM})},
\end{align*}
where the last equality holds since $i\geq n^{2\neweps}/M$ and thus  $\Capp_M(\{u_m\}_{m=0}^{h-1}) \leq \frac{M^2}{\delta n} = o(\frac{\delta iM^2}{2n})$, which is a lower bound for $\Capp_M((Y_m)_{m=0}^{iM}))$ by \cref{lem:cap union bound}. This proves the claim.
\end{proof}

Using the Laplacian random walk representation, we can now tightly control the probability that the next LERW segment will have good capacity.

\begin{corollary}\label{cor:capacity tail}
Let $T \subset G_n$ be a subgraph such that $|T| \leq n^{1/2+\neweps}$ and let $Y$ be a LERW trajectory started at some $u \in V(G_n) \setminus T$ and killed when it hits $T$. Then for each $0 \leq i \leq \frac{n^{1/2+ \neweps}}{M}$, and any LERW trajectory $\{Y_m\}_{m=0}^{iM} = \{y_m\}_{m=0}^{iM}$ disjoint from $T$ we have
\begin{align*}
    &\prcond{(Y_m)_{m=iM}^{(i+1)M} \cap {T} = \emptyset \text{ and } \left|\Capp_M (Y[iM, (i+1)M)) - \frac{\alpha_n M^2}{n}\right| \geq  \frac{\alpha_n M^2}{n}n^{-\neweps/16}}{\{Y_m\}_{m=0}^{iM} = \{y_m\}_{m=0}^{iM}}{} \\
    &\leq \frac{4M^2}{\delta n}.
\end{align*}
\end{corollary}
\begin{proof}
Let $A$ be the set of simple paths $(u_m)_{m=0}^{M}$ not intersecting $T$, with $\left|\Capp_M (u[0,M)) - \frac{\alpha_n M^2}{n}\right| \geq  \frac{\alpha_n M^2}{n}n^{-\neweps/16}$ and with $u_0=Y_{iM}$. It follows from \eqref{eqn:LERW path factorisation} that for any $i \geq 1$,
\begin{align}\label{eqn:LERW RW relation event}
\begin{split}
\prcond{ (Y_m)_{m=iM+1}^{iM+M} \in A}{(Y_m)_{m=0}^{iM} = (y_m)_{m=0}^{iM}}{} &\leq \prstart{ (X_m)_{m=1}^{H} \in A}{Y_{iM}} \sup_{(u_m)_{m=1}^{M} \in A} \left\{C((y_m)_{m=0}^{iM}, T, (u_m)_{m=1}^{M}))\right\}.
\end{split}
\end{align}
As $\left(u_m\right)_{m=1}^M$ does not intersect $T$, by \cref{lem:hitting lem LB rough same event}, the supremum is at most $2$ for all sufficiently large $n$.
Therefore, by \cref{prop:RW cap}, we have that
\begin{align*}
&\prcond{\left|\Capp_M (Y[iM, (i+1)M)) - \frac{\alpha_n M^2}{n}\right| \geq  \frac{\alpha_n M^2}{n}n^{-\neweps/16} \text{ and } Y[iM, (i+1)M) \cap T = \emptyset}{(Y_m)_{m=0}^{iM} = (y_m)_{m=0}^{iM}}{} \\
&\leq \frac{4M^2}{\delta n}. \qedhere
\end{align*}
\end{proof}

In the next lemma we compute some hitting probabilities for a random walk.

\begin{lemma}\label{lem:hitting without intersections}
Let $X$ be a random walk on $G_n$. Let $T \subset G_n$ be a subgraph such that $|T| \leq n^{1/2+\neweps}$, and let $A$ be a connected subset of $T$. Then for any $0 \leq i \leq \frac{n^{1/2+ \neweps}}{M}$ , for any simple path $y[0, iM]$ on $G_n$ disjoint from $T$ and any $u\notin T\cup y[0,iM]$
\begin{align*}
   &\prstart{X[0,M] \cap A \neq \emptyset \pand \not \exists 0<j<\ell \leq \tau_{A} : X_j = X_{\ell} \pand \not \exists j \leq \tau_{A} : X_j \in y[0,iM] \cup (T \setminus A)}{u} \\
   &= \Capp_M(A) (1 + o(1)).
\end{align*}
\end{lemma}
\begin{proof}
\textbf{Upper bound.} Recall that $\tplus = \log_2^2(n)\tmix$. First note that by \cref{claim:capacity different start point}, we have that
\begin{align*}
   &\prstart{X[0,M] \cap A \neq \emptyset \pand \not \exists 0<j<\ell\leq \tau_{A} : X_j = X_{\ell} \pand \not \exists j \leq \tau_{A} : X_j \in y[0,iM]\cup (T \setminus A)}{u} \\
   &\leq \prstart{X[0,M] \cap A \neq \emptyset}{u} \\
   &\leq \Capp_M(A) + \frac{3|A|\log n \cdot \tmix}{\delta n} \leq \Capp_M (A) \left(1+\frac{\tplus}{\delta^2 M}\right).
\end{align*}
Here the final line follows because of \cref{lem:cap union bound}, which implies that $\Capp_M(A) \geq \frac{\delta M|A|}{2n}$ on the event $|T| \leq n^{1/2+\neweps}$.

\textbf{Lower bound.}
We first note that
\begin{align*}
    &\prstart{X[0,M] \cap A \neq \emptyset}{u} - \prstart{X[0,M] \cap A \neq \emptyset \pand \exists 0 < j < \ell < \tau_A: X_j = X_{\ell}}{u} \\
    &- \prstart{X[0,M] \cap A \neq \emptyset \pand \exists j < M: X_j \in y[0,iM] \cup (T \setminus A)}{u}.
    \\
    &\leq 
    \prstart{X[0,M] \cap A \neq \emptyset \pand \not \exists 0<j<\ell \leq \tau_A : X_j = X_{\ell} \pand \not \exists j \leq \tau_A : X_j \in y[0,iM] \cup (T \setminus A)}{u}
\end{align*}
We will now bound all three terms on the left hand side.
First, we lower bound $\prstart{X[0,M]\cap A \neq \emptyset}{u}$ by $\Capp_M (A) \left(1-\frac{3 \tplus}{\delta^2 M}\right)$ by \cref{claim:capacity different start point}.
For the second term, we upper bound it by the product of the probabilities for $X$ to self intersect in $M$ steps, and then to hit $A$ in another $M$ steps. This is upper bounded by
\begin{equation*}
    \frac{M^2}{\delta n}\cdot \frac{M|A|}{\delta n}.
\end{equation*}
The third term can be bounded by the probability to hit $y[0,iM]\cup (T \setminus A)$ in at most $M$ steps, and then to hit $A$ in at most another $M$ steps, which is upper bounded by
\begin{equation*}
    \frac{M|A|}{\delta n}\cdot \frac{M(iM+|T|)}{\delta n}.
\end{equation*}
We conclude that
\begin{align*}
    &\Capp_M(A)\left(1+O\left(\frac{\tplus}{M}\right)\right) \leq \Capp_M(A)\left(1-\frac{\tplus}{\delta^2 M}\right) - 2\cdot\frac{M|A|}{\delta n}\cdot \frac{M(iM+|T|)}{\delta n} \\&
    \leq \prstart{X[0,M] \cap A \neq \emptyset \pand \not \exists 0<j<\ell\leq \tau_A : X_j = X_{\ell} \pand \not \exists j \leq \tau_{A} : X_j \in y[0,iM] \cup (T \setminus A)}{u}.\qedhere
\end{align*}
\end{proof}

We now have estimates for all the quantities appearing in \eqref{eqn:LERW path factorisation}. We combine these in the next corollary.

\begin{corollary}\label{cor:hit A in next interval}
Let $T \subset G_n$ be a subgraph such that $|T| \leq n^{\neweps+ 1/2}$, and let $A \subseteq T$. Let $Y$ be a LERW on $G_n$ killed when it hits $T$. For each $\frac{n^{2\neweps}}{M} \leq i \leq \frac{n^{1/2+\neweps}}{M}$ and for any simple path $(y_m)_{m=0}^{iM}$ not intersecting $T$, satisfying
\[
\left|\Capp_M (y[jM, (j+1)M)) - \frac{\alpha_n M^2}{n}\right| \leq  \frac{\alpha_n M^2}{n}n^{-\neweps/16} \text{ for all } \frac{n^{2\neweps}}{M} \leq j < i,
\]
it holds that
\begin{align*}
    &\prcond{Y \text{ hits } T \text{ in time interval } [iM, (i+1)M) \text{ in set } A}{(Y_m)_{m=0}^{iM} = (y_m)_{m=0}^{iM}}{} \\
    &= \frac{\Capp_M(A)}{\Capp_M(T)} \left(\Capp_M (T) + \frac{\alpha_n iM^2}{n}\right) (1+o(1)).
\end{align*}
\end{corollary}
\begin{proof}
First note that it follows from \cref{cor:closeness UB deter} that 
\[
\Capp_M (y[0, iM)) = \frac{\alpha_n iM^2 (1+o(1))}{n}.
\]
Given $1 \leq H < M$, $(y_m)_{m=0}^{iM}$ and $T$, let $\Gamma_{y_{iM} \to A, T, H}$ denote the set of simple paths with $u_0=y_{iM}$ that first hit $T$ in the set $A$ and at time $H$, and avoid $(y_m)_{m=0}^{iM}$ until that time. We can then write, using \cref{lem:hitting lem LB rough same event}(b) and \cref{lem:hitting without intersections}:
\begin{align*}
&\prcond{Y \text{ hits } T \text{ in time interval } [iM, (i+1)M) \text{ in set } A}{(Y_m)_{m=0}^{iM} = (y_m)_{m=0}^{iM}}{}
 \\
&= \sum_{H < M} \sum_{\substack{(u_m)_{m=0}^H \in \\ \Gamma_{y_{iM} \to A, T, H}}} \prcond{Y[iM, iM+H] = (u_m)_{m=0}^H}{(Y_m)_{m=0}^{iM} = (y_m)_{m=0}^{iM}}{} \\
&=\sum_{H < M} \sum_{\substack{(u_m)_{m=0}^H \in \\ \Gamma_{y_{iM} \to A, T, H}}} C((y_m)_{m=1}^{iM}, T, (u_m)_{m=1}^{H}) \prstart{X[0, H] = (u_m)_{m=0}^H}{y_{iM}} \\
&=\sum_{H < M} \sum_{\substack{(u_m)_{m=0}^H \in \\ \Gamma_{y_{iM} \to A, T, H}}} \frac{[\Capp_M(T) + \Capp((Y_m)_{m=0}^{iM})](1+o(1))}{\Capp_M(T)} \prstart{X[0, H] = (u_m)_{m=0}^H}{y_{iM}} \\
&=\frac{[\Capp_M(T) + \Capp((Y_m)_{m=0}^{iM})](1+o(1))}{\Capp_M(T)} \sum_{H < M} \sum_{\substack{(u_m)_{m=0}^H \in \\ \Gamma_{y_{iM} \to A, T, H}}} \prstart{X[0, H] = (u_m)_{m=0}^H}{y_{iM}} \\
&=\frac{[\Capp_M(T) + \Capp((Y_m)_{m=0}^{iM})]}{\Capp_M(T)} \Capp_M(A)(1+o(1)).
\end{align*}
Here the final line follows by \cref{claim:capacity 	different start point}, \cref{cor:closeness UB deter} and  \cref{lem:hitting without intersections}.
\end{proof}

\begin{proof}[Proof of \cref{cor:main stick breaking ingredient}]
\begin{enumerate}[(1)]
    \item Fix some $T$ which is good, and fix some $C>0$. For every $i \leq C\beta_n \sqrt{n} / M$, let $E_{C,i}$ be the event that
\begin{equation*}
\left|\Capp_M (Y[iM, (i+1)M)) - \frac{\alpha_n M^2}{n}\right| \leq  \frac{\alpha_n M^2}{n}n^{-\neweps/16} \pand Y[iM,(i+1)M) \cap T = \emptyset.
\end{equation*}
Write $E_{prefix}$ for the event $\cap_{i\leq n^{2\neweps}/M} E_{C,i}$. Note that by  \cref{lem:dont hit too soon} and \cref{cor:capacity tail} we have that
\begin{equation*}
    \pr{E_{prefix}} = 1 - o(1).
\end{equation*}

Note that, by \cref{cor:capacity tail} and \cref{cor:hit A in next interval}, for any $i \geq n^{2\neweps} / M$, given $E_{prefix} \cap_{n^{2\neweps}/M \leq j \leq i} E_{C,j}$, using \cref{cor:hit A in next interval} we have
\begin{align*}
&\prcond{E_{C,i}}{ E_{prefix} \pand \cap_{n^{2\neweps}/M \leq j \leq i} E_{C,j} }{} \\
&= 1 - \left(\Capp_M (T) + \frac{\alpha_n iM^2}{n}\right) (1+o(1)) - O\left(\frac{4M^2}{\delta n}\right) = 1 - \left(\frac{\alpha_nM|T|}{n} + \frac{\alpha_n iM^2}{n}\right) (1+o(1)).
\end{align*}
Here the final line holds since $T$ is good, $i \geq \frac{n^{2\neweps}}{M}$ and our conditioning on $E_{prefix}\cap \cap_{n^{2\neweps}/M \leq j \leq i} E_{C,j}$.
Then, write $E_C$ for the event that
\begin{equation*}
    E_{prefix} \pand \{Y[0, C\beta_n\sqrt{n}] \cap T = \emptyset\} \pand \left( \cap_{i=n^{2\neweps}/M}^{C\beta_n \sqrt{n}/M} E_{C,i} \right),
\end{equation*}
we have that,
\begin{align*}
    &\prcond{E_C}{\Tkkn=T}{} =
    \\& \prcond{E_{prefix}}{\Tkkn = T}{}\cdot\\&\prod_{i=n^{2\neweps/M}}^{C\beta_n \sqrt{n} / M}\prcond{E_{C,i}}{E_{prefix} \pand \cap_{n^{2\neweps}/M \leq j < i}E_{C,j}}{} 
    \\ &=(1-o(1))\prod_{i=n^{2\neweps}/M}^{C\beta_n\sqrt{n}/M}
     \left(1 - \left(\frac{\alpha_nM|T|}{n} + \frac{\alpha_n iM^2}{n}\right) (1+o(1))\right) 
     \\& =(1-o(1))\exp\left(-\sum_{i=n^{2\neweps}/M}^{C\beta_n\sqrt{n}/M}(1-o(1))\alpha_n(M|T|/n + iM^2/n)\right) 
     \\&= (1-o(1))\exp\left(-\alpha_n\frac{C\beta_n\sqrt{n}}{2M}(2M|T|/n + C\beta_nM/\sqrt{n})\right) \\
     &=(1-o(1))\exp(-(C^2/2 + C|T|/(\beta_n\sqrt{n})).
\end{align*}
(Here in the last line we used that $\beta_n^2=\frac{1}{\alpha_n}$ by definition).
To conclude, note that
\begin{align}\label{eqn:no bad capacity}
    \prcond{Y[0, C \beta_n\sqrt{n}] \cap \Tkkn = \emptyset}{\Tkkn = T}{} - \prcond{E_C}{\Tkkn = T}{} \leq \frac{C\beta_n\sqrt{n}}{M} \cdot \frac{4M^2}{\delta n} = o(1) 
\end{align}
by \cref{cor:capacity tail}. On the event $|T| \leq B \sqrt{n}$, this can be written in the form $o_B(1) \prcond{E_C}{\Tkkn = T}{}$ using the estimate above.

\item It follows directly from \cref{cor:hit A in next interval} and \cref{def:good tree} that for any $i>n^{2\neweps}/M$, conditionally on $H_{\Tkkn} \in (iM, (i+1)M]$, we have that
\begin{align*}
\prcond{Y_{H_{\Tkkn}} \in A}{H_{\Tkkn} \in (iM, (i+1)M]}{} = \frac{|A|(1+o(1))}{|\Tkkn|},
\end{align*}
as required.
\item Given $\epsilon>0$, first choose $C<\infty$ so that the probability appearing in part (1) is at most $\epsilon$. Then, on the event $Y[0, C\beta_n\sqrt{n}] \cap T \neq \emptyset$, we have that the probability that $\Tkn$ is not good is upper bounded by $\epsilon + o(1)$ by \eqref{eqn:no bad capacity}. Since $\epsilon>0$ is arbitrary this gives the result.
\end{enumerate}
\end{proof}

\section{Proof of \cref{thm:main2}}\label{sctn:Wilson is stickbreaking}

In this section we prove \cref{thm:main2}. We start by using the estimates of the previous section to prove \cref{thm:pr04}. At the end of the section, we address the lower mass bound condition which completes the proof of \cref{thm:main2}.

\subsection{Proof of \cref{thm:pr04}}

In \cref{def:stick breaking} we defined how a sequence of trees can be constructed through a stick-breaking process. In what follows next we outline how, for any $k \geq 1$, Wilson's algorithm on $G_n$ can be used to give two sequences $(Y_i)_{i = 0}^n$ and $(Z_i)_{i = 0}^{n-1}$ such that $\SBk ((Y_0, Y_1, \ldots, Y_{k-1}), (Z_0, Z_1, \ldots, Z_{k-2} ))$ is equal to the subtree obtained after the first $k-1$ steps of Wilson's algorithm, and such that the ${k \choose 2}$ distances appearing in \cref{thm:pr04} therefore match those between the points $(Y_0, Y_1, \ldots, Y_{k-1})$ in the stick-breaking construction.

    Let $G_n$ be a graph on $n$ vertices and recall the definition of $\beta_n$ from \eqref{eq: def of beta n}. We will define a stick-breaking process $(Y_i^n)_{i = 0}^{n}, (Z_i^n)_{i = 0}^{n-1}$ which arises from Wilson's Algorithm on $G_n$. To ease notation, we shall remove the superscript and begin with $Y_0=0, Z_0=0$. We choose an ordering of the vertices of $G$, denoted by $\{v_1, \ldots, v_n\}$. Then, at the first step, we sample the $\UST$ path (using Wilson's algorithm) from $v_2$ to $v_1$. We denote this path by $T^{(2)}_n$, and let $Y_1$ be the length of this path divided by $\beta_n \sqrt{n}$. For every vertex $z$ on this path we say that $z$ was added at the first step. Let $k\geq 2$ and assume that we sampled $\Tkn$ and $Z_0,\ldots,Z_{k-2}$ and $Y_0,\ldots, Y_{k-1}$. For the $k^{th}$ step, take $v_{k+1}$ and sample (again, using Wilson's algorithm) its path to $\Tkn$. Denote this path $P_k$ and set $\Tkkkn = P_k \cup \Tkn$. For every vertex in $P_k \setminus \Tkn$, we say that it was added on the $k^{th}$ step. Let $Y_k = Y_{k-1} + \frac{|P_k|}{\beta_n\sqrt{n}}$. In order to define $Z_{k-1}$, first let $z$ be the vertex at which $P_k$ hits $\Tkn$. If $z$ is of the form $v_m$ for some $m$, set $Z_{k-1} = Y_{m-1}$. Otherwise, let $m<k$ be the step at which $z$ was added. Then, $Y_{m-1} \leq Z_{k-1} \leq Y_{m}$ and the exact value of $Z_{k-1}$ is 
\begin{equation*}
    Z_{k-1} = Y_{m} - \frac{d(z,v_{m+1})}{\beta_n \sqrt{n}}.
\end{equation*}

Furthermore, this way we can define a function $I$ that identifies every $v\in \Tkn$ with a point in $[0,\infty)$. If $v$ was added at the $m^{th}$ step then we set $I(v) = Y_m - \frac{d(v,v_{m+1})}{\beta_n \sqrt{n}}$.

Throughout this section, we also let $(Y_i')_{i \geq 0}$ and $(Z_i')_{i \geq 0}$ be the analogous quantities for stick-breaking of the CRT, sampled as described in \cref{prop:CRT stick breaking}.

We will use the following claim. Recall the definition of ``good" from \cref{def:good tree}.

\begin{claim}\label{cl: discrete and uniform are close}
	Assume that $\Tkkn$ is good and that $|\Tkkn| \in [C^{-1}\sqrt{n}, C\sqrt{n}]$. Let $I_{\max} = \max_{v \in \Tkkn}I(v)$. Let $Y$ be a LERW started from $v_k$, and let $H_{\Tkkn}$ be the time at which $Y$ hits $\Tkkn$. Let $j \leq n^{1/2+\neweps}/M$ and let $\Pb_{d,j}$ be the measure on $[0, I_{\max}]$ defined by
	\begin{equation*}
	   \prstart{I(v)}{d,j} = \prcond{Y_{H_{\Tkkn}} = v}{H_{\Tkkn} \in [j\eps \beta_n \sqrt{n}/2, (j+1)\eps\beta_n \sqrt{n}/2)}{} \ \ \ \forall v \in \Tkkn.
	\end{equation*}
	Then, for every $\eps>0$ there exists $N\in \NN$ such that for all $n>N$ and for all $j \leq n^{1/2+\eps}/M$, the Prohorov distance between the measure $\Pb_{d,j}$
	and the uniform probability measure on $[0,I_{\max}]$ is less than $\eps$.
\end{claim}

\begin{proof}
We assume that $C/\beta_n$ is larger than $\eps$, otherwise $I_{\max}$ is smaller than $\eps$ and there's nothing to prove. We also assume wlog that $\epsilon<1$. Decompose $[0,I_{\max}]$ into intervals of size $\eps$ by writing $[0,I_{\max}] = \cup_{i \leq \lfloor I_{\max}/\eps \rfloor} [i\eps, \min\{I_{\max}, (i+1)\eps\}]$. Fix $j$, write $\Pb_d$ in place of $\Pb_{d,j}$ and
denote by $\Pb_u$ the uniform measure on $[0,I_{\max}]$. Note that every interval $I \subset [0, I_{\max}]$ of length $\eps$ can be identified with the union of at most $k$ connected subsets of $\Tkkn$ such that the sum of their lengths is $\epsilon \beta_n \sqrt{n}$ (which is much larger than $n^{3\epsilon}$). By discarding those that are of length less than $n^{3\epsilon}$ we can apply \cref{cor:main stick breaking ingredient}(2) to the remaining subsets (by decomposing them and $I$ into smaller intervals if necessary) to deduce that
\begin{equation*}
    \prstart{I}{d} = \frac{\eps}{I_{\max}}(1+o(1)), \quad \prstart{I}{u} = \frac{\eps}{I_{\max}}.
\end{equation*}
Now take $N$ large enough such that the $o(1)$ error is bounded by $\eps$.
Then, take some set $A$ in $[0,I_{\max}]$ and let $I_A$ be the set of intervals of the form $[i\eps, (i+1)\eps)$ intersecting $A$.
Now, we have that
\begin{equation*}
|I_A|\frac{\eps}{I_{\max}} - \eps \leq |I_A|\frac{\eps}{I_{\max}}(1+o(1))  \leq \prstart{I_A}{d} \leq |I_A|\frac{\eps}{I_{\max}}(1+o(1)) \leq |I_A|\frac{\eps}{I_{\max}} + \eps 	
\end{equation*}
 and
 \begin{align*}
 	\prstart{A}{d} \leq \prstart{I_A}{d} \leq |I_A|\frac{\eps}{I_{\max}} + \eps \leq \prstart{I_A}{u} + \eps \leq \prstart{A^\eps}{u} + \eps, 
 \end{align*}
 \begin{align*}
 	\prstart{A}{u} \leq \prstart{I_A}{u} = |I_A|\frac{\eps}{I_{\max}} \leq \prstart{I_A}{d} + \eps \leq \prstart{A^\eps}{d} + \eps.
 \end{align*}
Hence the Prohorov distance between these two measures is at most $\eps$.
\end{proof}

The main claim of this section is now as follows.

\begin{claim}\label{claim:coupling of stick breaking}
For every $\epsilon>0$ and $k \geq 1$ there exists $N$ such that for all $n>N$ we can couple the stick-breaking process for the CRT and for the $\UST$ such that $|Y_i - Y_i'| \leq \epsilon$ for all $0 \leq i\leq k-1$ and $|Z_i- Z_i'| \leq \epsilon$ for every $0 \leq i\leq k-2$ with probability at least $1-\epsilon$.
\end{claim}
\begin{proof}
We prove the claim by induction. Clearly when $k=1$ (i.e. for $T^{(1)}_n$) the statement holds trivially since the tree is a single point and $Y_0 = Y_0' = 0$ by construction. Moreover since a tree consisting of a single vertex is always good and since $Z_0 = Z_0' = 0$, it also follows directly from \cref{lem: proh close rv}, \cref{prop:CRT stick probs} and \cref{cor:main stick breaking ingredient} that the statement holds for $k=2$ as well. 

Now fix $k\geq 3$ and suppose that the claim holds for all $m<k$. We will now show that the claim holds also for $k$. That is, suppose that for every $\eps>0$ there exists $N$ large enough such that for all $n>N$ we can successfully couple $\Tkkn$ with the CRT. It suffices to show that for every $\eps>0$, there exists $0 < \zeta < \epsilon/8$ such that if we condition on a successful coupling of the previous step with parameter $\zeta$, then we can couple $(Y_{k-1}, Z_{k-2})$ with $(Y'_{k-1}, Z'_{k-2})$ such that $|Y_{k-1} - Y'_{k-1}| < \eps$ and $|Z_{k-2} - Z'_{k-2}| < \eps$ with probability at least $1-\epsilon/2$.

To this end, let $\zeta>0$ (its precise value will be chosen later) and suppose we have successfully coupled $(Y_i)_{i \leq k-2}$ and $(Z_i)_{i \leq k-3}$ with $(Y_i')_{i \leq k-2}$ and $(Z_i')_{i \leq k-3}$ as in the statement of the claim with parameter $\zeta$.
Note that it follows directly by iterating Point 3 of \cref{cor:main stick breaking ingredient} that $\Tkkn$ is good with probability at least $1-\epsilon/3$ for all sufficiently large $n$. Moreover, it therefore also follows from \cref{lem: size of stick} that $0<g(\eps) \leq Y_{k-2} \leq  f(\eps)$ with probability at least $1-\epsilon/3$, for some functions $f$ and $g$ where $g(\eps)>0$ and $f(\epsilon)<\infty$. Hence we can assume that we coupled $\Tkkn$ with the CRT, that $\Tkkn$ is good and that $g(\eps) \leq Y_{k-2} \leq f(\epsilon)$.

Under the coupling, we can write  $\frac{|\Tkkn|}{\beta_n \sqrt{n}} = Y_{k-2} = Y'_{k-2} + \eps'$ where $\epsilon' \in [-\zeta, \zeta]$. Therefore, since $\Tkkn$ is good, it follows from \cref{cor:main stick breaking ingredient}(1) with $B=f(\epsilon)$ that for any $C \in (0, \infty)$,
\begin{align*}
        \prcond{Y_{k-1} - Y_{k-2} > C }{\Tkkn}{} &= \exp\left\{- \frac{(C + Y'_{k-2} + \eps')^2-(Y'_{k-2} + \eps')^2}{2} \right\} + o(1) \\
        &=\exp\left\{- \frac{(C + Y'_{k-2})^2-(Y'_{k-2})^2}{2} \right\} \left(1+e^{-C\epsilon'} - 1 \right) + o(1),
        \end{align*}
so
\begin{align}\label{eq: distance between Ys}
\left| \prcond{Y_{k-1} - Y_{k-2} > C }{\Tkkn}{} - \exp\left\{- \frac{(C + Y'_{k-2})^2-(Y'_{k-2})^2}{2} \right\} \right| \leq |1-e^{-C\epsilon'}|e^{\frac{-C^2}{2}} + o(1) \leq C\zeta e^{\frac{-C^2}{2}} + o(1).
     \end{align}
The first term on the right hand side goes to $0$ as $\zeta\to 0$ uniformly over $C>0$.
By \cref{lem: proh close rv}, there exists $\eta$ depending on $\eps$ such that if $f(\eps) \leq Y_{k-2} \leq g(\eps)$ and the right-hand side of \eqref{eq: distance between Ys} is smaller than $\eta$, then we can couple $Y_{k-1} - Y_{k-2}$ and $Y'_{k-1} - Y'_{k-2}$ such that the probability that they are $\eps/4$ close to one another is at least $1-\eps/4$. When this happens, by the triangle inequality, we have that $|Y_{k-1} - Y'_{k-1}| < \eps / 2$. We therefore choose $\zeta$ small enough (and smaller than $\eps / 8$) and $n$ large enough such that the right-hand side is smaller than this $\eta$.

However, we note that $Z_{k-2}$ is not independent of $Y_{k-1}$ and we are required to couple the pair $(Y_{k-1}, Z_{k-2})$ with $(Y'_{k-1}, Z'_{k-2})$.
To do so, we will decompose $\RR^+$ into intervals of length $\eps/2$, that is, we write $\RR^+ = \bigcup_{j=0}^{\infty} I_j$ where $I_j = [j\eps/2, (j+1)\eps/2)$. Let $M_{k-1}$ (respectively $M'_{k-1}$) be the unique $j$ such that $Y_{k-1} \in I_j$ (respectively $Y'_{k-1} \in I_j$).
By \cref{lem: proh close rv} and the discussion above, there exists a coupling of $M_{k-1}$ and $M'_{k-1}$ such that the difference between them is at most $1$ with probability $1-\eps/8$. By \cref{lem: size of stick}, with probability at least $1-\eps/8$ we have that $M'_{k-1} \leq n^{1/2+\eps}-1$ for $n$ large enough (and then so is $M_{k-1}$). 

Then, given $M_{k-1}$, we sample $Z_{k-2}$ according to its conditional law. By \cref{cl: discrete and uniform are close}, when $n$ is large enough, for every $j\leq n^{1/2+\neweps}$, conditionally on $M_{k-1}=j$ we have that the Prohorov distance between $Z_{k-2}$ and a uniform random variable on $[0,Y_{k-2}]$ is at most $\zeta$. By \cref{cl: two uniforms are close}, the Prohorov distance between a uniform random variable on $[0,Y_{k-2}]$ and $Z'_{k-2}$ (recall that, given $Y_{k-2}'$, $Z'_{k-2}$ is independent of $Y'_{k-1}$ and hence of $M'_{k-1}$) is at most $\zeta$.  Therefore, the Prohorov distance between $Z'_{k-2}$ and $Z_{k-2}$ conditionally on $M_{k-1} = j$ is at most $2\zeta$. Since $\zeta < \eps/8$, it follows that we can couple the pairs $(Y_{k-1}, Z_{k-2})$ and $(Y'_{k-1}, Z'_{k-2})$ such that $|Y_{k-1} - Y'_{k-1}| < \eps$ and $|Z'_{k-2} - Z_{k-2}| < \eps$ with probability at least $1 - \eps/2$, as required.
\end{proof}

\begin{corollary}
For every $\epsilon>0$ and $k \geq 1$ there exists $N$ such that for all $n>N$ we can couple the stick-breaking process for the CRT and for the $\UST$ such that, with probability at least $1-\epsilon$, it holds for all $0 \leq i,j \leq k$ that
\[
|d(y_i, y_j) - d'(y_i', y_j')| \leq \epsilon.
\]

\end{corollary}
\begin{proof}
Take $\eta>0$ and $k \geq 1$. We verify that there is a coupling such that each of the conditions of \cref{prop:stick breaking close} hold with high probability.

For the first condition note that, by \cref{claim:coupling of stick breaking}, we can couple the stick-breaking process for the CRT and for the $\UST$ such that $|Y_i - Y_i'| \leq \eta$ for all $0 \leq i\leq k$ and $|Z_i- Z_i'| \leq \eta$ for every $0 \leq i\leq k-1$ with probability at least $1-\eta$ for all sufficiently large $n$. For the second condition, note that it follows from \cref{prop:CRT stick breaking} that we can choose $\delta=\delta(\eta, k)>0$ such that $|Z_i' - Y_j'| \geq 3\eta$ for all $i \leq k-1,j \leq k$ with probability at least $1-\delta$, and such that $\delta \downarrow 0$ as $\eta \downarrow 0$.

Therefore, it follows from \cref{prop:stick breaking close} that under this coupling, it holds with probability at least $1-\eta-\delta$ that $\sup_{1 \leq i,j \leq k}|d(y_i, y_j) - d'(y_i', y_j')| \leq 2k\eta$. Given $\epsilon>0$, we can therefore choose $\eta>0$ small enough that $2k\eta+\delta<\epsilon$ in order to deduce the claim as stated.
\end{proof}

\begin{proof}[Proof of \cref{thm:pr04}]
For $k \geq 1$, let $D^{(k)}_n$ denote the matrix of distances between $k$ uniform points in $\UST (G_n)$. Let $D^{(k)}$ denote the analogous matrix for the CRT.

We showed that for any $k\geq 1$ and any $\epsilon>0$, we can couple $\Tkn$ and the CRT so that $||D^{(k)}_n - D^{(k)}||_{\infty} < \epsilon$ with probability at least $1-\epsilon$ for all sufficiently large $n$. Thus we have that $||D^{(k)}_n - D^{(k)}||_{\infty}$ converges to $0$ in probability and therefore $D^{(k)}_n$ converges to $D^{(k)}$ in distribution, which is equivalent to the statement of \cref{thm:pr04}. 
\end{proof}

\subsection{Lower mass bound}

To strengthen the convergence obtained in \cref{thm:pr04} to GHP convergence (and therefore prove \cref{thm:main2}), it suffices to verify that \cref{prop:GP plus LMB gives GHP}(b) holds. Therefore, in our setting, it is enough to show the following.

\begin{claim}[Lower mass bound]\label{clm:lmb}
Let $(G_n)_{n \geq 1}$ be a dense sequence of deterministic graphs satisfying the assumptions of \cref{thm:main2}. For each $n \geq 1$, let $\T_n$ be a uniformly drawn spanning tree of $G_n$. Denote by $d_{\T_n}$ the corresponding graph-distance on $\T_n$ and by $\mu_n$ the uniform probability measure on the vertices of $\T_n$. Then, for any $c>0$ and any $\eta > 0$ there exists some $\eps>0$ such that for all $n\in \NN$
\begin{equation*}
    \pr{\exists v\in \T_n: |B_{\T_n}(v, c\sqrt{n})| \leq \eps n}{} \leq \eta.
\end{equation*}
\end{claim}

The results of \cite{ANS2021ghp} establish the lower mass bound for a sequence $(G_n)_{n \geq 1}$ such that $|G_n| = n$ for all $n$, satisfying the following three conditions.

		\begin{enumerate}
			\item There exists $\theta < \infty$ such that $\displaystyle \sup_n \sup_{x \in G_n} \sum_{t=0}^{\sqrt{n}} (t+1) p_t(x,x) \leq \theta$.
			\item There exists $\alpha > 0$ such that $\tmix (G_n) = o (n^{\frac{1}{2} - \alpha})$ as $n \to \infty$.
			\item $G_n$ is transitive for all $n$.
		\end{enumerate}

For a graph sequence satisfying the assumptions of \cref{thm:main2}, note that the second condition is immediately satisfied by \cref{claim:mixing time expander}. The first condition is also satisfied since $p_t(x,x) \leq \frac{1}{\delta n}$ for all $x \in G_n$ and all $t \geq 1$.

However, we would like to relax the condition that $G_n$ is transitive and instead require only that the graphs are \emph{balanced}; that is, that there exists a constant $D< \infty$ such that $$\frac{\max_{v \in G_n} \deg v}{\min_{v \in G_n} \deg v} \leq D$$ for all $n$. As remarked at the end of \cite{ANS2021ghp}, it is straightforward to extend the proof of the lower mass bound to this setting by carrying the constant $D$ through all of the computations in \cite{ANS2021ghp}; we do not provide the details as they are not illuminating. Under the assumptions of \cref{thm:main2}, we can take $D=\delta^{-1}$ so this easily verifies \cref{clm:lmb} and therefore \cref{prop:GP plus LMB gives GHP}(b). Moreover, \cref{thm:pr04} ensures that \cref{prop:GP plus LMB gives GHP}(a) is also fulfilled. \cref{thm:main2} therefore follows directly.

\section{Proof of \cref{thm:main1} and \cref{cor:mainrandom}}\label{sctn:graphon conv}
Recall from the introduction that that a graphon $W$ is \textbf{non-degenerate} if the function
\begin{equation*}
    \deg_W(x) := \int_{[0,1]}W(x,y)dy
\end{equation*}
is defined and strictly positive for every $x\in[0,1]$, and that a non-degenerate graphon $W$ is \textbf{connected} if for every measurable $A\subset[0,1]$ we have that
\begin{equation*}
    \int_{A}\int_{A^C}W(x,y)dxdy > 0.
\end{equation*}

In order to verify \cref{thm:main1} as consequence of \cref{thm:main2}, we need to verify that under the assumptions of \cref{thm:main1}, the graph sequence is an expander sequence and that $\alpha_n \to \alpha_W$.

We start with the first of these. Recall the definition of a $\gamma$-expander sequence is given in \cref{def:expander}.

\begin{claim}\label{claim:graphon is an expander}
Let $W:[0,1]^2 \to [0,1]$ be a connected graphon and let $G_n$ be a sequence of weighted graphs with minimal degree at least $\delta n$ converging in cut-distance to $W$. Then there exists $\gamma = \gamma (W, \delta)>0$ such $(G_n)_{n \geq 1}$ is a $\gamma$-expander sequence.
\end{claim}
\begin{proof}
Take $U \subset G_n$. We split the proof into two cases depending on whether $|U| \geq \frac{1}{2}\delta n$ or not.

\textbf{Case 1: $|U| \leq \frac{1}{2}\delta n$.} Since $G_n$ has minimal degree at least $\delta n$ and the maximal weight of every edge is $1$, it follows that there is a total weight of $\frac{1}{2}\delta n$ emanating from every vertex leading to $V(G) \setminus U$, so that 
\[
w(U, V(G) \setminus U) \geq \frac{1}{2}|U|\delta n \geq \frac{1}{2}\delta|U| (V(G) - |U|).
\]

\textbf{Case 2: $|U| > \frac{1}{2}\delta n$.} By \cref{lem:connectivity}, there exists a constant $\beta = \beta(W,\delta)$ such that for every set $U$ with $\frac{\delta}{2} \leq \mu(U) \leq 1/2$ we have
\begin{equation*}
    \int_{U}\int_{U^C} W(x,y) > \beta.
\end{equation*}
In particular, since $G_n$ converges to $W$ this implies that there exists $N>\infty$ such that for all such $U$ and all $n \geq N$,
\begin{equation*}
    \frac{1}{n^2} w(U, V(G) \setminus U) = \int_{U}\int_{U^C} W_n(x,y) > \frac{\beta}{2}.
\end{equation*}
Since $|U|(V(G)-|U|) \leq n^2$ trivially, this implies that $e(U, V(G) \setminus U) \geq \frac{\beta}{2}|U|(V(G)-|U|)$.
\end{proof}

We now turn to verifying the convergence $\alpha_n \to \alpha_W$. Recall that in \cref{sctn:RW estimates} we defined
\begin{equation*}
\alpha_n = \frac{n\estart{\Capp_M (X[0, n^{\neweps/2}))}{\pi}}{Mn^{\neweps/2}}.
\end{equation*}
where $X$ is a RW on $G_n$ started from stationarity, and showed that under some assumptions, the sequence 
\begin{equation}\label{eqn:GHP conv}
\left(\UST (G_n), \frac{\sqrt{\alpha_n}d_n}{\sqrt{n}}, \mu_n\right) \overset{(d)}{\to} (\T, d, \mu)
\end{equation}
with respect to the GHP topology. We also let $U_{\pi_n}$ denote a random stationary vertex of $G_n$, and define 
\[
\tilde{\alpha}_n = n\E{\pi_n(U_{\pi_n})}.
\]
In fact it is more convenient to deal with $\tilde{\alpha}_n$ rather than $\alpha_n$. This is sufficient as we show in the following claim (we write the proof for completeness, but really it follows directly just from linearity of expectation and \cref{cor:closeness UB deter}).

\begin{claim}\label{claim:alpha n equivalence}
Let $(G_n)_{n \geq 0}$ be a sequence of weighted graphs on $n$ vertices with minimal degree $\delta n$. Let $\alpha_n$ and $\tilde{\alpha}_n$ be defined as above. Then $\alpha_n = \tilde{\alpha}_n(1+o(1))$ as $n \to \infty$.
\end{claim}
\begin{proof}
By the Bonferroni inequalities and linearity of expectation, and letting $Z$ denote an independent RW started from stationarity, we can write (recalling also from \cref{lem:cap union bound} that $\Capp_M(U_{\pi}) \geq \frac{M\delta}{2n}$ deterministically):
\begin{align*}
 \left|\E{\Capp_M(X[0,n^{\eps/2}))} - n^{\eps/2} \E{\Capp_M(U_{\pi})}\right| &= \left|\E{\Capp_M(X[0,n^{\eps/2}))} - \sum_{i=0}^{n^{\eps/2}} \E{\Capp_M(X_{i})}\right| \\
&\leq \prstart{\exists 0 \leq t_1 < t_2 \leq M-1:Z_{t_1} \cap X[0,n^{\eps/2}) \neq \emptyset \text{ and } Z_{t_2} \cap X[0,n^{\eps/2}) \neq \emptyset}{\pi} \\
&\leq \left(\frac{Mn^{\eps/2}}{\delta n}\right)^2 \leq \frac{2Mn^{\eps/2}}{\delta^3 n} \cdot n^{\eps/2} \E{\Capp_M(U_{\pi})}.
\end{align*}
Similarly, then note that, since $\pi(v) \geq \frac{\delta}{n}$ for all $v \in G_n$ deterministically,
\begin{align*}
\left|\E{\Capp_M(U_{\pi})} - M \E{\pi(U_{\pi})} \right| &= \left|\E{\Capp_M(U_{\pi})} - \sum_{i=0}^{M-1} \prstart{Z_i = U_{\pi}}{\pi} \right| \\
&\leq \prstart{\exists 0 \leq t_1 < t_2 \leq M-1:Z_{t_1} = Z_{t_2} = U_{\pi}}{\pi} \leq \left(\frac{M}{\delta n}\right)^2 \leq \frac{M}{\delta^3 n} \cdot M\E{\pi(U_{\pi})}.
\end{align*}
To conclude, we combine to get that
\begin{align*}
\alpha_n = \frac{n}{Mn^{\eps/2}} \E{\Capp_M(X[0,n^{\eps/2}))} = \frac{n}{M} \E{\Capp_M(U_{\pi})}(1+o(1))=n \E{\pi(U_{\pi})}(1+o(1)) = \tilde{\alpha}_n(1+o(1)),
\end{align*}
as required.
\end{proof}
It therefore follows that the convergence of \eqref{eqn:GHP conv} holds with the sequence $(\tilde{\alpha}_n)_{n \geq 1}$ in place of $({\alpha}_n)_{n \geq 1}$. To prove main convergence theorem, it is therefore sufficient to show that, under the assumptions of \cref{thm:main1},
\begin{equation}\label{eqn:tilde alpha convergence}
    \tilde{\alpha}_n \to \alpha_W,
\end{equation}
where $\alpha_W$ is as in \eqref{eqn:alphaW}.

\begin{remark}\label{rmk:regular alpha=1}
Note that $\tilde{\alpha}_n$ is $1$ when $G_n$ is regular, so clearly \eqref{eqn:tilde alpha convergence} will entail that $\alpha_W=1$ for a regular graph sequence.
\end{remark}

Our next goal is to show the following.

\begin{claim}\label{claim:alphan convergence}
Let $W:[0,1]^2 \to [0,1]$ be a connected graphon and let $G_n$ be a sequence of graphs with stationary measures $\pi_n$ converging in cut-distance to $W$. Then, 
\begin{equation*}
\tilde{\alpha}_n := n\estart{\pi_n(v)}{\pi_n} \to \frac{1}{\left(\int_{[0,1]^2}W(x,y)dxdy\right)^2} \cdot \int_{[0,1]} \left(\int_{[0,1]}W(x,y)dy\right)^2 dx.
\end{equation*}
\end{claim}
\begin{proof}
We will begin with showing that 
\begin{equation}\label{eqn:graphon degree convergence}
    \frac{1}{n^2}\sum_{v\in V} \deg_{G_n}(v) \to \left(\int_{[0,1]^2}W(x,y)dxdy\right).
\end{equation}
Indeed, as $G_n \to W$ in the cut-distance, there exist $\phi_n$ measure-preserving automorphisms such that the graphon representations $W_n$ of $G_n$ satisfy
\begin{equation*}
    \sup_{S,T \in \mathcal{B}([0,1])}
    \left|\int_{S}\int_{T} W^{\phi_n}_n(x,y) - W(x,y) dx dy\right| \to 0.
\end{equation*}
We henceforth write $W_n$ in place of $W^{\phi_n}_n$. In particular, choosing $S=T=[0,1]$ we obtain that
\begin{equation*}
    \int_{[0,1]}\int_{[0,1]} W_n(x,y) dx dy \to \int_{[0,1]}\int_{[0,1]}W(x,y).
\end{equation*}
However, by the definition of a graphon representation of a graph given in \cref{sctn:graphon background}, we also have that
\begin{equation*}
    \int_{[0,1]}\int_{[0,1]} W_n(x,y) dx dy = \frac{1}{n^2}\sum_{v\in V} \deg_{G_n}(v),
\end{equation*}
which establishes \eqref{eqn:graphon degree convergence}.
Next, we will show that 
\begin{equation}\label{eqn:graphon degree squared convergence}
    \frac{1}{n^3} \sum_{v\in G_n} \deg_{G_n}(v)^2 \to \int_{[0,1]} \left(\int_{[0,1]}W(x,y)dy\right)^2 dx
\end{equation}
Note that in every graphon representation $W_n$ of $G_n$ we have that
\begin{equation*}
    \frac{1}{n^3}\sum_{v\in G_n}\deg_{G_n}(v)^2 = \frac{1}{n}\sum_{v\in G_n}\left(\frac{\deg_{G_n}(v)}{n}\right)^2 = \frac{1}{n} \sum_{v\in G_n} \left(\int_{[0,1]}W_n(x_v,y) dy\right)^2,
\end{equation*}
where $x_v$ is some point in $[0,1]$ corresponding to $v$. Moreover, in the notation of \cref{sctn:graphon background}, it follows from the construction given there that
\begin{equation*}
    \frac{1}{n} \sum_{v\in G_n} \left(\int_{[0,1]}W_n(x_v,y) dy\right)^2 = \sum_{i=1}^n \int_{I_i}\left(\int_{[0,1]}W_n(x,y) dy\right)^2 dx = \int_{[0,1]}\left(\int_{[0,1]}W_n(x,y) dy\right)^2 dx.
\end{equation*}
To establish \eqref{eqn:graphon degree squared convergence}, it thus suffices to prove that
\begin{equation*}
    \int_{[0,1]}\left(\int_{[0,1]}W_n(x,y) dy\right)^2 dx \to  \int_{[0,1]} \left(\int_{[0,1]}W(x,y)dy\right)^2 dx.
\end{equation*}
In other words, writing $\deg_W$ and $\deg_{W_n}$ for the corresponding normalized degree functions of the graphons $W$ and $W_n$ as defined in \eqref{eqn: deg of graphon}, we need to show that
\begin{equation}\label{eqn:degree squared convergence 2}
    \int_{[0,1]} \left(\deg_{W_n}(x)^2 - \deg_{W}(x)^2\right) dx \to 0.
\end{equation}
As $\deg_{W_n}$ and $\deg_{W}$ are measurable functions, we have that the set $\{x \in [0,1] : \deg_{W_n}(x) > \deg_W(x)\}$ is measurable. Denote this set by $S$. We have that
\begin{align*}
    \int_{S} \deg_{W_n}(x)^2 - \deg_{W}(x)^2 dx &= \int_{S} (\deg_{W_n}(x) - \deg_{W}(x))(\deg_{W_n}(x) + \deg_{W}(x)) dx \\&\leq 2\int_{S}(\deg_{W_n}(x) - \deg_{W}(x)) dx = 2\int_{S}\int_{[0,1]}W_n(x,y) - W(x,y) dy dx \to 0.
\end{align*}
By symmetry and considering $S^c$ we similarly have that
\begin{equation*}
    \int_{S^c} \deg_{W_n}(x)^2 - \deg_{W}(x)^2 dx \to 0,
\end{equation*}
from which we conclude that \eqref{eqn:degree squared convergence 2} and therefore \eqref{eqn:graphon degree squared convergence} hold.
Finally, given \eqref{eqn:graphon degree convergence} and \eqref{eqn:graphon degree squared convergence}, note that
\begin{align*}
\tilde{\alpha}_n &= n \sum_{v \in G_n} \left(\frac{\deg(v)}{\sum_{v\in G_n}\deg(v)}\right)^2 = n\cdot\frac{1}{\left(\sum_{v\in G_n}\deg(v)\right)^2}\sum_{v \in G_n}(\deg(v)^2) \\&= \left(n^2 \cdot \frac{1}{\sum_{v \in G_n} \deg(v)}\right)^2\cdot\left(\sum_{v\in G_n}\frac{1}{n^3}\deg(v)^2\right) \to \frac{1}{\left(\int_{[0,1]^2}W(x,y)dxdy\right)^2} \cdot \int_{[0,1]} \left(\int_{[0,1]}W(x,y)dy\right)^2 dx,
\end{align*}
as required.
\end{proof}

\begin{proof}[Proof of \cref{thm:main1}]
We showed in \cref{claim:graphon is an expander} that under the assumptions of \cref{thm:main1}, the graph sequence in question is an expander sequence, so that \cref{thm:main2} applies. In \cref{claim:alpha n equivalence} and \cref{claim:alphan convergence}, we showed that the sequence $\alpha_n$ appearing in the conclusion of \cref{thm:main2} converges to $\alpha_W$ as $n \to \infty$, exactly as required.
\end{proof}

\cref{cor:mainrandom} is a direct consequence of \cref{thm:main1} and \cref{lem:random graphons are close}.

\bibliographystyle{abbrv}
\bibliography{biblio}

\end{document}